\documentclass[11pt]{amsart}

\usepackage{graphicx}        
\usepackage{multicol}        
\usepackage[bottom]{footmisc}
\usepackage{amscd,amsmath,amssymb,amsfonts}
\usepackage[cmtip, all]{xy}
\usepackage{enumerate}

\usepackage[linktocpage=true]{hyperref}

\newcommand{\chapter}{\section}

\newtheorem{theorem}{Theorem}[section]
\newtheorem{proposition}[theorem]{Proposition}
\newtheorem{lemma}[theorem]{Lemma}
\newtheorem{corollary}[theorem]{Corollary}

\theoremstyle{definition}
\newtheorem{definition}[theorem]{Definition}
\theoremstyle{remark}
\newtheorem{remark}[theorem]{Remark}
\newtheorem{remarks}[theorem]{Remarks}
\newtheorem{ex}[theorem]{Example}

\numberwithin{equation}{section}
\theoremstyle{plain}
\newtheorem*{introthm}{Theorem}

\newcommand{\CH}{{\rm CH}}
\newcommand{\mc}{\mathcal}

\newcommand{\red}{{\rm red}}

\newcommand{\Hom}{{\rm Hom}}

\newcommand{\Spec}{{\rm Spec\,}}

\newcommand{\Char}{{\rm char}}

\newcommand{\supp}{{\rm supp}\,}
\newcommand{\0}{\emptyset}

\newcommand{\sC}{{\mathcal C}}
\newcommand{\sD}{{\mathcal D}}
\newcommand{\sE}{{\mathcal E}}

\newcommand{\sO}{{\mathcal O}}

\newcommand{\sR}{{\mathcal R}}
\newcommand{\sS}{{\mathcal S}}

\newcommand{\sU}{{\mathcal U}}

\newcommand{\sW}{{\mathcal W}}
\newcommand{\sX}{{\mathcal X}}
\newcommand{\sY}{{\mathcal Y}}

\newcommand{\A}{{\mathbb A}}

\newcommand{\F}{{\mathbb F}}

\newcommand{\N}{{\mathbb N}}
\renewcommand{\P}{{\mathbb P}}
\newcommand{\Q}{{\mathbb Q}}
\newcommand{\R}{{\mathbb R}}

\newcommand{\Z}{{\mathbb Z}}

\newcommand{\Zar}{{\text{\rm Zar}}}

\newcommand{\colim}{{\mathop{\rm colim}}}

\renewcommand{\dim}{{\operatorname{\rm dim}}} 
 
\newcommand{\xr}[1]{\xrightarrow{#1}}

\newcommand{\Proj}{{\operatorname{Proj}}}

\newcommand{\Sym}{{\operatorname{Sym}}}

\newcommand{\Tor}{{\operatorname{Tor}}}

\newcommand{\Gr}{{\operatorname{\rm Gr}}}

\newcommand{\Deg}{\text{deg}}

\newcommand{\ind}[1]{}

\newcommand{\inp}[1]{}

\newcommand{\OO}{\sO}

\begin{document}
\setcounter{tocdepth}{1}

\title{Torsion orders of complete intersections}

\author{Andre Chatzistamatiou}
\author{Marc Levine}


\begin{abstract}
By a classical method due to Roitman, a complete intersection $X$ of sufficiently small degree admits a rational decomposition of the diagonal. This means that some multiple of the diagonal by a positive integer $N$, when viewed as a cycle in the Chow group, has support in $X\times D\cup F\times X$, for some divisor $D$ and a finite set of closed points $F$. The minimal such $N$ is called the torsion order. We study lower bounds for the torsion order following the specialization method of Voisin, Colliot-Th\'el\`ene and Pirutka. We give a lower bound for the generic complete intersection with and without point. Moreover, we use methods of Koll\'ar and Totaro to exhibit lower bounds for the very general complete intersection.   
\end{abstract}
\maketitle

\tableofcontents
\section*{Introduction}
Decomposition of the diagonal has played a prominent role in recent progress on stable rationality questions. For a rationally connected variety over a field $k$, there is a minimal integer $\Tor_k(X)\geq 1$ such that the multiple of the diagonal $\Tor_k(X)\cdot \Delta_X$, when viewed in the Chow group of $X\times X$, is supported in $X\times D\cup F\times X$, for some divisor $D$ and some finite set of closed points $F$. We will call $\Tor_k(X)$ the \emph{torsion order} of $X$; it is a stable birational invariant which equals $1$ if $X$ is stably rational and in general gives an upper bound on the exponent of the unramified cohomology of $X$.  This invariant is also studied by Kahn \cite{Kahn}.  In a proper flat family the torsion order of a fiber divides the torsion order of the generic fiber (see Lemma~\ref{lem:specialization} for the precise statement). One can  thus deduce a non-trivial torsion order from a non-trivial torsion order of a cleverly chosen degeneration.

This method was pioneered by Voisin \cite{Voisin}. It was significantly simplified and applied by Colliot-Th\'el\`ene and Pirutka to show the non-rationality of a very general quartic fourfold \cite{CTP} by using a degeneration to a classical example of Artin and Mumford (after a ``universally $\CH_0$-trivial'' resolution of singularities \cite[Definitions 1.1, 1.2]{CTP}), which is a unirational but non-rational variety. The non-trivial $2$-torsion in its Brauer group forces non-triviality of the torsion order (in fact, it implies that the torsion order is even). Totaro \cite{Totaro} used Voisin's method combined with work of Koll\'ar \cite{Kollar} to improve Koll\'ar's non-rationality results for hypersurfaces in {\it loc.~cit}. Roughly speaking, Totaro showed how, for large enough degree, a general hypersurface of even degree degenerates to an inseparable degree 2 cover in characteristic 2 whose resolution of singularities can be shown to support non-vanishing differential forms. As for the Brauer group, action of correspondences (and the fact that the singularities of the degeneration are ``not too bad'') shows divisibility of the torsion order by $2$.   
   
In this paper we study the torsion order of complete intersections in projective space. A classical result by Roitman, which we recall in Proposition \ref{prop:UpperBound}, establishes an upper bound stating that a complete intersection $X$ of multi-degree $(d_1,\dots,d_r)$ in $\P^{n+r}_k$ (over any field $k$) with $\sum_{i=1}^r d_i \leq n+r$ satisfies $\Tor_k(X) \mid \prod_{i=1}^r (d_i!)$. Our first result is a lower bound for a generic complete intersection. 

\begin{introthm}[Theorem~\ref{theorem:main1}, Corollary \ref{corollary-with-point}]
Let $\mathcal{Y}:=\prod_{i=1}^r \P(H^0(\P^{n+r}_k,\OO(d_i))^{\vee})$, and let $\mathcal{X}\subset \mathcal{Y}\times \P^{n+r}_k$ be the incidence variety 
$$
\mathcal{X}=\{(f_1,\dots,f_r,x)\in \mathcal{Y}\times \P^{n+r}_k \mid f_1(x)=\dots=f_r(x)=0\}.
$$ 
We denote by $K$ the quotient field of $\mathcal{Y}$, and let $X/K$ be the generic fiber of the family $\mathcal{X}\xr{} \mathcal{Y}$. For an integer $d\ge1$, let $d!^*$ be  be the least common multiple of the integers $1,\ldots, d$. The following holds: 
\begin{enumerate}[i)]
\item $\Tor_K(X)$ is divisible by $\prod_{i=1}^r d_i!^*$,  
\item $\Tor_{K(X)}(X\otimes_K K(X))$ is divisible by $\frac{\prod_{i=1}^r d_i!^*}{d_1\cdots d_r}$.
\end{enumerate}
\end{introthm}

The invariant which detects divisors of the torsion order in the first part of theorem is the index of a variety, that is, the image of the Chow group of zero cycles via the degree map. The index of $X/K$ is given by $d_1\cdots d_r$. Divisibility of the torsion order by other integers of the form $i_1\cdot\ldots\cdot i_r$ with $1\le i_j\le d_j$ is shown by degeneration to a union of complete intersections with lower degrees and using induction.   

We also consider the generic cubic hypersurface with a line, and use Theorem~\ref{theorem:main1} to show that this has torsion order exactly 2 (Example~\ref{ex:line}). We show the existence of a cubic threefold over $K=\Q_p((x))$ or $K=\F_p((t))((x))$, having a $K$-point and torsion order divisible by  2 (Example~\ref{ex:small}); more generally, we construct  examples of cubic hypersurfaces of dimension $n$ over a field $K=k((x))$, where $k$ is a field of characteristic zero and $u$-invariant at least $n+1$, which have a $K$-point and for which 2 divides the torsion order. This last series of examples is  taken over from \cite{CTNew}, with the kind permission of the author, and it gives an improvement over a construction in an earlier version of this paper, which relied on Rost's degree formula.  We should mention that other examples of this kind  already exist in the literature,   see for example \cite[Th\'eor\`eme 1.21]{CTP}, where cubic threefolds over a $p$-adic field with non-zero torsion order are constructed, as well as examples over $\F_p((x))$ \cite[Remarque 1.23]{CTP}; both examples have a rational point. 

Our second result concerns the torsion order of very general complete intersections over algebraically closed fields of characteristic zero. The idea of the proof is as in the papers of Koll\'ar and Totaro. We are able to generalize the results on the Hodge cohomology of the degeneration in characteristic $p$ to Hodge--Witt cohomology. In this way we can establish results on divisibility by powers of $p$. 

\begin{introthm}[Theorem~\ref{thm-main-alg-closed}] Let $k$ be an algebraically closed field of characteristic zero. Let $X\subset \P^{n+r}_k$ be a very general complete intersection of multi-degree  $d_1, d_2, \dots, d_r$ such that $d':=\sum_{i=1}^r d_i \leq n+r$ and $n\geq 3$. Let $p$ be a prime,  $m\geq 1$, and suppose
\begin{equation*}
d_i \geq p^{m}\cdot \left\lceil {\frac{n+r+1-d'+d_i}{p^m+1}} \right\rceil 
\end{equation*}
for some $i$, where $\left\lceil\; \right\rceil$ denotes the ceiling function. Moreover, we assume that $p$ is odd or $n$ is even. Then $p^m| \Tor_k(X)$.
\end{introthm}
For example, if $\sum_{i=1}^r d_i = n+r$ and $n\geq 3$, which is the extreme case, then $d_i| \Tor_k(X)$ if $d_i$ is odd or $n$ is even. For hypersurfaces and $m=1$, the theorem is due to Totaro, and we give a short proof of the straight-forward generalization to complete intersections and the case $m=1$  in Theorem~\ref{thm:LowerBound}. We should mention that our Theorem~\ref{thm:LowerBound} and Theorem~\ref{thm-main-alg-closed} are actually a bit stronger, in that we prove the same divisibility result for the torsion orders of  level $n-2$ (see below), which automatically divide the torsion orders described above.

The paper is divided into seven sections. Section~\ref{sec:TorsionOrder} contains the definition and basic properties of the torsion order. Following a suggestion of Claire Voisin, we consider decompositions of the diagonal of higher ``niveau level'' and the associated torsion invariants;  we also describe some elementary specialization results.  In section~\ref{sec:trivial} we recall from Colliot-Th\'el\`ene and Pirutka the notion of a universally $\CH_0$-trivial morphism and a related notion, that of a totally $\CH_0$-trivial morphism.   Behavior under a combination of degeneration and modification by a birational totally $\CH_0$-trivial morphism, which is the basic tool used for divisibility results, is the focus of section~\ref{sec:degen}; in this section we follow \cite{CTP} and extend their specialization results to cover decompositions of higher level.  We recall Roitman's theorem in section~\ref{sec:Roitman} and discuss the case of the generic complete intersection in section~\ref{sec:generic}. We recall Totaro's arguments leading to the divisibility results for the torsion order of a very general complete intersection  in section~\ref{sec:LowerBound1} and conclude by proving our refined version in section~\ref{sec:LowerBound2}.

We would like to thank the referees very much for thoroughly reading the paper and suggesting improvements. We are especially grateful to the referee who suggested the statement and proof of Lemma~\ref{lemma-advanced-technology-for-Wm}. This result enabled us to improve an earlier version of our Theorem~\ref{thm-main-alg-closed} to the statement on higher torsion orders mentioned above. We are also grateful to Jean-Louis Colliot-Th\'el\`ene, who very kindly allowed us to include some of the results  of his paper \cite{CTNew}. This led to a new result (Lemma~\ref{lem:Spec3}) on specialization of decompositions of the diagonal, derived from \cite[Lemma 2.2]{CTNew}, and  Example~\ref{ex:small} mentioned above, a version of which appears as \cite[Th\'eor\`eme 2.4]{CTNew}.  

\section{Torsion orders}\label{sec:TorsionOrder}
Let $k$ be a field and $X$ a  $k$-scheme of finite type.  If $A$ is a presheaf on $X_\Zar$, we let 
\[
A(X(i)):=\colim_FA(X\setminus F)
\]
where $F$ runs over all closed subsets of $X$ with $\dim_kF\le i$. We extend this notation to products, defining for a presheaf $A$ on $(X\times_kY)_\Zar$
 \[
 A(X(i)\times Y(j))=\colim_{F, G}A((X\setminus F)\times(Y\setminus G)).
 \]
For example, the contravariant functoriality of the classical Chow groups  for open immersions  allows us to apply this notation to $A(X):=\CH_n(X)$ for some $n$.

Let $k$ be a field with algebraic closure $\bar k$. We say that a  finite type $k$-scheme $X$ is {\em generically reduced}  if $X$ is reduced at each generic point. We call a reduced finite type $k$-scheme $X$ {\em separable} over $k$ if the total quotient ring $k(X)$ is a product of separably generated field extensions of $k$. For $X$ an arbitrary finite type $k$-scheme, call $X$ separable over $k$ if $X_{red}$ is so. We note that for $X$ generically reduced and separable over $k$, $X\times_k\bar{k}$ is also generically reduced. A closed subset $D$ of a finite type $k$-scheme $X$ is called {\em nowhere dense} if $D$ contains no generic point of $X$.

\begin{definition}\label{def:Basic}  Let $k$ be a field and let $X$ be a reduced proper  $k$-scheme of pure dimension $d$ over $k$. \\
1. For $i=0,1, 2,\ldots$, the $i$th torsion order of $X$, $\Tor^{(i)}_k(X)\in \N_+\cup\{\infty\}$, is the order of the image of the diagonal $\Delta_X\subset X\times_kX$ in $\CH_d(X(i)\times X(d-1))$. We write $\Tor_k(X)$ for $\Tor^{(0)}_k(X)$ and call this the torsion order of $X$. \\
2. Suppose $X$ is separable over $k$.  For $1\le i<j\le3$, let $p_{ij}:X\times_kX\times_kX\to X\times_kX$ denote the projection on the $i$th and $j$th factors, and let $\Delta_{ij}\subset X\times_kX\times_kX$ denote the pullback $p_{ij}^{-1}(\Delta_X)$. Consider the Cartesian diagram
\[
\xymatrix{
X_{k(X\times X)}\ar[r]^{\tilde j}\ar[d]&X\times_kX\times_kX\ar[d]^{p_{23}}\\
\Spec k(X\times X)\ar[r]_i&X\times_kX.
}
\]
Let $\eta_1-\eta_2\in \CH_0(X_{k(X\times_k X)})$ denote the class of the pullback $\tilde{j}^*(\Delta_{12}-\Delta_{13})$. The {\em generic torsion order} of $X$, $g\Tor_k(X)\in \N_+\cup\{\infty\}$, is the order of  $\eta_1-\eta_2$ in $\CH_0(X_{k(X\times X)})$.\\
3.  We say that $X$ admits a decomposition of the diagonal of order $N$ and level $i$ if there is a nowhere dense closed subset $D$, a closed subset $Z$ of $X$ with $\dim_kZ\le i$ and cycles $\gamma$, $\gamma'$ on $X\times X$, with $\gamma$ supported in $X\times D$, $\gamma'$ supported in $Z\times X$  and with 
\[
N\cdot[\Delta_X]= \gamma'+\gamma
\]
in $\CH_d(X\times_kX)$.   \\
4. Suppose $X$ is geometrically integral. For an integer $N\ge1$, we say that $X$ admits a decomposition of the diagonal of order $N$ if there is a 0-cycle $x$ on $X$, a proper closed subset $D$ of $X$ and a dimension $d$ cycle $\gamma$ on $X\times_kX$, supported in $X\times D$, such that
\[
N\cdot[\Delta_X]=x\times X+\gamma
\]
in $\CH^d(X\times_kX)$. We say that $X$ admits a $\Q$-decomposition of the diagonal if $X$ admits a decomposition of order $N$ for some $N$, and that $X$ admits a $\Z$-decomposition of the diagonal if $X$ admits  a decomposition of the diagonal of order $1$. \\
5. Let $\deg:CH_0(X)\to \Z$ be the degree map. For $X$ smooth and integral, the {\em index} of $X$ is the positive generator $I_X$ of the subgroup $\deg \CH_0(X)\subset \Z$. Equivalently, $I_X$ is the g.c.d. of all degrees $[k(x):k]$ as $x$ runs over closed points of $X$. We extend the definition of the index to  proper, integral, separable $k$-schemes $Y$ by defining $I_Y$ to be the g.c.d. of all degrees $[k(y):k]$ as $y$ runs over closed points of the smooth locus $Y_{sm}$ of $Y$ (which is dense in $Y,$ as $Y$ is separable over $k$).
\end{definition}

\begin{remarks}\label{rem:1stRemarks} 1.  Suppose $X$ has pure dimension $d$ over $k$ and is geometrically integral. Since the only dimension $d$ cycles $\gamma_{(0)}$ on $X\times X$,  supported on $Z_{(0)}\times X$ with $Z_{(0)}\subset X$ a dimension zero closed subset are of the form $\gamma_{(0)} =x\times X$ for some 0-cycle $x$ on $X$, a decomposition of the diagonal of order $N$ and level 0 is the same as decomposition of the diagonal of order $N$.\\
2. We extend the definition of $\Tor_k^{(i)}(X)$ to all proper, equi-dimensional $k$-schemes by setting $\Tor^{(i)}_k(X):=\Tor^{(i)}_k(X_{red})$.  \\
3. We will often use an equivalent formulation of  Definition~\ref{def:Basic}(3), namely, that $X$ admits a decomposition of the diagonal of order $N$ and level $i$ if there is a closed subset $D$  containing no generic point of $X$ and a closed subset $Z$ of $X$ with $\dim_kZ\le i$  such that 
\[
N\cdot j^*[\Delta_X]=0
\]
in $\CH_d((X\setminus Z)\times_k(X\setminus D))$, where $j:(X\setminus Z)\times_k(X\setminus D)\to X\times_kX$ is the inclusion. This equivalence follows from the localization sequence
\[
\CH_d(Z\times_kX\cup X\times_kD)\xrightarrow{i_*}
\CH_d(X\times_kX)\xrightarrow{j^*} \CH_d((X\setminus Z)\times_k(X\setminus D))\to0
\]
and the surjection
\[
\CH_d(Z\times_kX)\oplus\CH_d(X\times_kD)\to \CH_d(Z\times_kX\cup X\times_kD).
\]
4. Decompositions of the diagonal  for {\em smooth} proper $k$-varieties have been considered in \cite{BS, CTP, Totaro} and by many others. Here we have extended the definition to proper, equi-dimensional, but not necessarily smooth $k$-schemes.
\end{remarks}

\begin{lemma} \label{lem:Basic} Let $X$ be a proper  $k$-scheme of pure dimension $d$ over $k$. \\
1. If  $\Tor^{(i)}_k(X)$ is finite then so $\Tor^{(i+1)}_k(X)$ and in this case, $\Tor^{(i+1)}_k(X)$ divides $\Tor^{(i)}_k(X)$.\\
2. $X$ admits a decomposition of the diagonal of order $N$ and level $i$ if and only if $\Tor^{(i)}_k(X)$ divides $N$; if $X$ is geometrically integral, then $X$ admits a decomposition of the diagonal of order $N$  if and only if $\Tor_k(X)$ divides $N$ and $X$ does not admit a $\Q$-decomposition of the diagonal if and only if $\Tor_k(X)=\infty$.\\
3.  Suppose $X$ is smooth over $k$ and geometrically integral. If $\Tor_k(X)$ is finite then so is $g\Tor_k(X)$ and $g\Tor_k(X)$ divides $\Tor_k(X)$.\\
4. Suppose $X$ is  separable over $k$ and let $L\supset k$ be a field extension. If $\Tor^{(i)}_k(X)$ is finite, then so is $\Tor^{(i)}_L(X_L)$ and in this case, $\Tor^{(i)}_L(X_L)$  divides $\Tor^{(i)}_k(X)$. If $L$ is finite over $k$ then $\Tor^{(i)}_k(X)$ is finite if and only if  $\Tor^{(i)}_L(X_L)$ is finite and in this case   $\Tor^{(i)}_k(X)$ divides $[L:k]\cdot\Tor^{(i)}_L(X_L)$. The  corresponding statements hold replacing $\Tor^{(i)}$ with $g\Tor$.\\
5. $X$ admits a decomposition of the diagonal of level $i$ and order $N$ if and only if there is a closed subset $Z\subset X$ of dimension $\le i$ such that the pull-back of $\Delta_X$ to $(X\setminus Z)\times_k\Spec k(X)$ via the inclusion
\[
(X\setminus Z)\times_k\Spec k(X)\to X\times_k X
\]
has order dividing $N$ in $\CH_0((X\setminus Z)\times_k\Spec k(X))$.
\end{lemma}

\begin{proof} (1) follows from the existence of the restriction homomorphism
\[
\CH_d((X\setminus F)\times_k(X\setminus D))\to 
\CH_d((X\setminus F')\times_k(X\setminus D))
\]
for $F\subset F'$.  (2) follows  from the localization sequence for $\CH_*(-)$, as in Remark~\ref{rem:1stRemarks}(3). 
 
 For (3),  suppose 
\[
N\cdot[\Delta_X]=x\times X+\gamma
\]
in $\CH_d(X\times_kX)$ for $x$ and $\gamma$ as in Definition~\ref{def:Basic}. Since $X$ is smooth and proper, we have for every field extension $F$ of $k$, the action of $\CH_d(X_F\times_FX_F)$ on $\CH_n(X_F)$ as correspondences (see \cite{Fulton}), that is, for $\alpha\in \CH_d(X_F\times_FX_F)$ and $\rho\in \CH_n(X_F)$, one has the well-defined element
\[
\alpha^*(\rho):=p_{1*}(p_2^*\rho\cdot \alpha).
\]
Acting by the correspondence $N\cdot\Delta^*_{X_{k(X\times X)}}$ on $\CH_0(X_{k(X\times_kX)})$ gives
\[
N\cdot (\eta_1-\eta_2)=x-x=0
\]
and thus  $g\Tor_k(X)$ divides $N$. Applying (2) gives (3). 

For  (4), the first assertion follows by applying the pull-back in $\CH_d$ for $X_L\times_LX_L\to X\times_kX$ and using (2). The second part  follows by applying the pushforward map $\CH_d(X_L\times_L X_L)\to 
\CH_d(X\times_k X)$ and using (2), and the assertion for $g\Tor_k(X)$ follows similarly  by applying the pushforward map $\CH_d(X_{L(X\times_kX)})\to \CH_d(X_{k(X\times X)})$. 

The last assertion (5) follows from the identity
\[
\CH_0((X\setminus Z)\times_k\Spec k(X))=\varinjlim_{D\subset X} \CH_d((X\setminus Z)\times_k(X\setminus D))
\]
where the limit is over all closed $D\subset X$ containing no generic point of $X$. 
\end{proof}

\begin{remark} We have restricted our attention to proper $k$-schemes for the definitions of torsion orders and decompositions of the diagonal. Even though the definitions would make sense for non-proper equi-dimensional $k$-schemes, a naive extension is probably not useful. Possibly replacing Chow groups with Suslin homology would make more sense: following Lemma~\ref{lem:Basic}, one could define $\Tor^{(i)}(X)$ for an equi-dimensional finite type $k$-scheme as the order of the restriction of $\Delta_X$ to $X\times_k\Spec k(X)$ in the quotient group 
\[
\lim_{\overrightarrow{Z\subset X}}H_0^{Sus}(X\times_k\Spec k(X))/im(H_0^{Sus}(Z\times_k\Spec k(X)))
\]
where $Z\subset X$ runs over all closed subsets of dimension at most $i$. We will not investigate properties of these torsion orders for non-proper $k$-schemes here.
\end{remark}

Here is the first in a series of elementary but useful specialization lemmas.

\begin{lemma}\label{lem:specialization}  Let $\sO$ be a noetherian regular local ring $f:\sX\to \Spec \sO$ a  proper flat morphism, with $\sX$ equi-dimensional over $\Spec\sO$  of  relative dimension $d$, $X\to \Spec K$ the generic fiber, $Y\to \Spec k$ the special fiber. We suppose that, for each $z\in\Spec\sO$, the fiber $\sX_z$ is generically reduced and separable over $k(z)$. Fix an integer $i$. \\
1. If $\Tor^{(i)}_K(X)$ is finite, then so is $\Tor^{(i)}_k(Y)$, and $\Tor^{(i)}_k(Y)$ divides $\Tor^{(i)}_K(X)$.\\
2. If $g\Tor_K(X)$ is finite, then so is $g\Tor_k(Y)$, and $g\Tor_k(Y)$ divides $g\Tor_K(X)$.\\
3. Let $\bar{k}$ and $\bar{K}$ be the respective algebraic closures of $k$ and $K$, and suppose either $K$ has characteristic zero, or that $\sO$ is excellent. If $\Tor^{(i)}_{\bar{K}}(X_{\bar{K}})$ is finite, then so is $\Tor^{(i)}_{\bar{k}}(Y_{\bar{k}})$, and $\Tor^{(i)}_{\bar{k}}(Y_{\bar{k}})$ divides $\Tor^{(i)}_{\bar{K}}(X_{\bar{K}})$.
\end{lemma}
\begin{proof} We use the definition of $\CH_d(X(i)\times X(d-1))$ as a limit to reduce to making computations in groups of the form $\CH_d((X\setminus Z)\times(X\setminus D))$ where $Z, D$ are closed subsets of $X$ with $\dim Z\le i$, $\dim D\le d-1$. We may stratify $\Spec\sO$ by regular closed subschemes $Z_0\subset\ldots\subset Z_r=\Spec\sO$, with $Z_i$ of Krull dimension $i$. This gives us the DVRs $\sO_i:=\sO_{Z_i, Z_{i-1}}$ and the restriction of $\sX$ to $\sX_i\to\Spec\sO_i$.  Regarding the proof of (3), if the original local ring $\sO$ has characteristic zero quotient field, we may stratify $\Spec\sO$ as above so that that each DVR $\sO_i$ has characteristic zero quotient field, and if $\sO$ is excellent, so are each of the $\sO_i$. Proving the result for each of the families $\sX_i$ gives the result for $\sX$, which reduces us to the case of a DVR $\sO$. 

In this case, suppose we have a relation
\begin{equation}\label{eqn:leveli}
N\cdot\Delta_X = 0
\end{equation}
in $\CH_d((X\setminus Z)\times (X\setminus D))$,  with $\dim_KZ\le i$ and $D$ nowhere dense. Taking the closures $\bar{Z}$ and $\bar{D}$ in $\sX$,  and letting $Z_0=Y\cap \bar{Z}$, $D_0=Y\cap \bar{D}$, we have the specialization homomorphism (see for example \cite[6.3.7]{Fulton})
\[
sp:\CH_d((X\setminus Z) \times_K(X\setminus D))  \to \CH_d((Y\setminus Z_0)\times_k(Y\setminus D_0)) 
\]
associated to the family 
\[
\sX\times_\sO\sX\setminus \bar{Z}\times \sX\cup \sX\times \bar{D}\to\Spec\sO. 
\]
Note that, as $\sO$ is a DVR, the closure $\bar{Z}$ is equi-dimensional over $\Spec \sO$, and thus $\dim_k Z_0\le i$; similarly, $D_0$ is nowhere dense in $Y$. Since  $\sX\to\Spec\sO$ is flat and the fibers are generically reduced, we have 
\[
sp(\Delta_X)=\Delta_Y
\]
in $\CH_d((Y\setminus Z_0)\times_k(Y\setminus D_0))$,
so applying $sp$ to \eqref{eqn:leveli}  proves (1). 

The proof of (2) is a similar specialization argument. Indeed,  we reduce as before to the case of a DVR $\sO$. Due to the generic separability assumption, there is a dense open subscheme $\sU$ of $\sX\times_\sO\sX$ that is smooth over $\Spec \sO$,  with special fiber dense in $Y\times_kY$. If now $\tau$ is a generic point of $Y\times_kY$, let $\sR$ be the local ring $\sO_{\sU, \tau}$. Then $\sR$ is a DVR and we may consider the $\sR$-scheme $\sX\otimes_\sO\sR\to \Spec\sR$. The quotient field $F$ of $\sR$ is one of the field factors of $k(X\times_KX)$ and the residue field $\mathfrak{f}$ of $\sR$ is the factor of $k(Y\times_kY)$ corresponding to $\tau$. Let $\eta_i^X$, $\eta^Y_i$, $i=1,2$ denote the images of the ``generic'' points used to define $g\Tor_K(X)$, resp. $g\Tor_k(Y)$  in  $\CH_0(\sX_F)$, resp. $\CH_0(Y_\mathfrak{f})$. Applying the specialization homomorphism
\[
sp:\CH_0(\sX_F)\to \CH_0(Y_\mathfrak{f})
\]
to a relation $N\cdot(\eta^X_1-\eta^X_2)$ in  $\CH_0(\sX_F)$ shows that $N\cdot(\eta^Y_1-\eta^Y_2)=0$ in $\CH_0(Y_\mathfrak{f})$ for each generic point $\tau$, and thus  $g\Tor_k(Y)$ divides $N$.

For (3), we note that there is a finite extension $L$ of $K$ so that 
\[
\Tor^{(i)}_{\bar{K}}(X_{\bar{K}})=\Tor^{(i)}_L(X_L) =\Tor^{(i)}_F(X_F)
\]
 for all finite extensions $F$ of $L$. Since either $K$ has characteristic zero or $\sO$ is excellent, the normalization $\sO^N$ of $\sO$ in $L$ is a semi-local principle ideal ring, finite over $\sO$ (the characteristic zero case follows from \cite[Chap.~V, Thm.~7]{ZS} and the excellent case follows from \cite[Theorem 78]{Matsumura}). Thus,  after replacing $\sO$ with the localization $\sO'$ of $\sO^N$ at a maximal ideal, and replacing $\sX$ with $\sX':=\sX\otimes_\sO\sO'$, we may assume that $\Tor^{(i)}_K(X)=\Tor^{(i)}_{\bar{K}}(X_{\bar{K}})$. Since  $\Tor^{(i)}_{\bar{k}}(Y_{\bar{k}})$ divides $\Tor^{(i)}_k(Y)$ by Lemma~\ref{lem:Basic}(4),  (3) follows from (1).
 \end{proof}

A global version of Lemma~\ref{lem:specialization}(3) follows by an argument using Hilbert schemes and Chow varieties. See \cite[Theorem~1.1 and Prop.~1.4]{Voisin} or \cite[Appendix B]{CTP} for similar statements.

\begin{corollary}\label{cor:specialization}  Let $p:\sX\to B$ be a flat, equi-dimensional and projective family over a  scheme $B$ of finite type over a field $k$ and let $b_0$ be a point of $B$. We suppose that each geometric fiber of $p$ is generically reduced.  Fix an integer $i\ge0$. Then there is a countable union of closed subsets $F=\cup_{i=1}^\infty F_i$ with $b_0\not\in F$ such that for all $b\in B\setminus F$, the geometric fiber $\sX_{\overline{k(b)}}$  satisfies
$\Tor^{(i)}(\sX_{\overline{k(b_0)}})\ |\ \Tor^{(i)}(\sX_{\overline{k(b)}})$. Here we use the convention that $N | \infty$ for all $N\in \N_+\cup\{\infty\}$ and $\infty | N\Rightarrow N=\infty$.  
\end{corollary}

\begin{proof}  Let $d$ be the relative dimension of $\sX$ over $B$. For a positive integer $M$, let $\sS(M)$ be the set  of $b\in B$ such that $M$ does not divide $\Tor^{(i)}(\sX_{\overline{k(b)}})$.  Taking $M=\Tor^{(i)}(\sX_{\overline{k(b_0)}})$ and $F=\sS(M)$, it suffices to show that $\sS(M)$ is a countable union of closed subsets of $B$. 

We first show that $\sS(M)$ is closed under specialization. Indeed, if we have a specialization $b\leadsto \bar{b}$ with $b\in \sS(M)$, then there is an excellent DVR $\sO$ and a morphism $\Spec\sO\to B$ with $b$ the image of the generic point of $\Spec\sO$ and $\bar{b}$ the image of the closed point. Indeed,  let $C$ be the closure of $b$ in $B$, blow-up $\Spec \sO_{C,\bar{b}}$ along $\bar{b}$, normalize to obtain a normal  scheme $\pi:T\to \Spec\sO_{C,\bar{b}}$ of finite type over  $\sO_{C,\bar{b}}$, choose a generic point $t$ of the Cartier divisor $\pi^{-1}(\bar{b})$ on $T$ and take $\sO:=\sO_{T,t}$. The local ring $ \sO_{C,\bar{b}}$ is excellent since $C$ is of finite type over a field, and the operations used in constructing $\sO$ from  $\sO_{C,\bar{b}}$ all preserve excellence (see \cite[Chapters 12, 13]{Matsumura}). Pulling back $\sX$ to $\Spec\sO$, it follows from Lemma~\ref{lem:specialization} and Lemma~\ref{lem:Index}(3) that $\bar{b}$ is also in $\sS(M)$. 

Since  $\sS(M)$ is closed under specialization, it suffices to show that, for each affine open subscheme $U$ of $B$, $\sS(M)\cap U$ is a countable union of closed subsets of $U$. Thus, we may assume that $B$ is affine, and that $\sX$ is a closed subscheme of $B\times\P^n_k$ for some $n$, with $p:\sX\to B$ the restriction of the projection.

By standard Hilbert scheme arguments, there is a projective $B$-scheme $q:\sY\to B$ such that the geometric points of $\sY$ consists of triples $(b,  Z, D)$, with $b$ a geometric point of $B$,   $Z\subset \sX_{\overline{k(b)}}$ a closed subscheme of dimension $j\le i$ and $D\subset \sX_{\overline{k(b)}}$ a  closed subscheme of dimension $<d$, and with $Z$ and $D$ having fixed Hilbert polynomials (chosen in advance). Similarly, using Chow varieties, there is a projective $B$-scheme $r:\sW\to B$ whose geometric points consists of  triples $(b, W^+, W^-)$ with  $W^+, W^-\subset \sX_{\overline{k(b)}}\times_{\overline{k(b)}}\sX_{\overline{k(b)}}\times\P^1$ dimension $d+1$ effective cycles of some fixed bi-degrees  (chosen in advance). $\sW$ contains the open subscheme $\sW^0$ of triples  $(b, W^+, W^-)$ such that both $W^+$ and $W^-$ have no component contained in  $\sX_{\overline{k(b)}}\times_{\overline{k(b)}}\sX_{\overline{k(b)}}\times\{0,\infty\}$. 

Fix an integer $N>0$. In $\sY\times_B\sW^0$ we have the closed subscheme $\sR_N$ whose geometric points consists of tuples $(b, Z, D, W^+, W^-)$ such that the cycle
\[
(\sX_{\overline{k(b)}}\times_{\overline{k(b)}}\sX_{\overline{k(b)}}\times0)\cdot(W^+-W^-)
\]
is supported in $Z\times \sX_{\overline{k(b)}}\times0\cup \sX_{\overline{k(b)}}\times D\times0$, and
\[
(\sX_{\overline{k(b)}}\times_{\overline{k(b)}}\sX_{\overline{k(b)}}\times\infty)\cdot(W^+-W^-)=N\cdot \Delta_X\times\infty.
\]

The image of $\sR_N$ under the projection $\sR_N\to B$ is a constructible subset of $B$. We vary the choice of $N$ over integers not divisible by $M$, and also vary over all choices of Hilbert polynomials (for dimension $\le i$ closed subschemes $Z$ and  closed subschemes $D$ of dimension $< d$) and all bi-degrees for the effective cycles $W^+$, $W^-$. As this set of choices is countable, it follows that $\sS(M)$ is a countable union of constructible subsets of $B$. As  $\sS(M)$ is closed under specialization, the proof is complete. 
\end{proof}

Next, we prove a modification of the specialization Lemma \ref{lem:specialization}. A related result may be found in \cite[Lemma 2.4]{Totaro}.

\begin{lemma}\label{lem:Spec2} Let $\sO$ be a discrete valuation ring with quotient field $K$ and residue field $k$. Let $f:\sX\to\Spec \sO$ be a flat and proper morphism of dimension $d$ over $\Spec \sO$ with generic fiber $X$ and special fiber $Y$. We suppose $Y$ is a union of closed subschemes, $Y=Y_1\cup Y_2$, with $Y_1$ and $Y_2$ having no common components,  and  that $X$ and $Y_1\setminus Y_2$ are generically reduced. Suppose in addition that $X$ admits a decomposition of the diagonal of order $N$ and level $i$. Then there is an identity in $\CH_d(Y_1\times_k Y_1)$
\[
N\Delta_{Y_1}=\gamma+\gamma_1+\gamma_2
\]
with $\gamma$ supported in $Z_1\times Y_1$ for some closed subset $Z_1\subset Y_1$ of dimension $\le i$, $\gamma_1$ supported on $Y_1\times D_1$, for some  nowhere dense closed subset $D_1\subset Y_1$, and $\gamma_2$  supported in $(Y_1\cap Y_2)\times Y_1$.
\end{lemma}

\begin{proof} We consider the (non-proper) $\sO$-scheme
$(\sX\setminus Y_2)\times_\sO(\sX\setminus Y_2)\to\Spec \sO$, closed subsets $Z, D$ of $X$ with $\dim_KZ\le i$, $D$ nowhere dense,   and a relation
\[
N\cdot[\Delta_X]= 0
\]
in $\CH_d((X\setminus Z)\times_K(X\setminus D))$, where $[\Delta_X]$ denotes the cycle class represented by the restriction of the diagonal. 

As in the proof of Lemma~\ref{lem:specialization}(1), we have closed subsets $Z_0, D_0$ of $Y_1^0:=Y_1\setminus Y_2$ with $\dim_kZ_0\le i$, $D_0$ nowhere dense,  and a specialization homomorphism
\[
sp:\CH_d((X\setminus Z)\times_K(X\setminus D)) \to \CH_d((Y_1^0\setminus Z_0)\times_k(Y_1^0\setminus  D_0)).
\]
As $X$ and $Y_1^0$ are  reduced at each generic point, it follows that $sp([\Delta_X])=[\Delta_{Y_1^0}]$, where $[\Delta_{Y_1^0}]$ is the cycle class of the restriction of the diagonal on $Y^0_1$. Applying $sp$ thus gives  the relation
\[
N\cdot [\Delta_{Y_1^0}] =0
\]
in $\CH_d((Y_1^0\setminus Z_0)\times (Y_1^0\setminus D_0))$. 

Let $Z_1:=\bar{Z}_0$ be the closures of $Z_0$  in $Y_1$, let $\bar{D}_0$ be the closure of $D_0$ in $Y_1$ and let $D_1=\bar{D}_0\cup (Y_1\cap Y_2)$. Using the localization sequence
\begin{multline*}
\CH_d(Z_1\times Y_1\cup    Y_1\times D_1\cup (Y_1\cap Y_2)\times Y_1)\to\\
\CH_d(Y_1\times_k Y_1)\to \CH_d((Y_1^0\setminus Z_0)\times (Y_1^0\setminus D_0))\to 0
\end{multline*}
and the surjection
\begin{multline*}
\CH_d(Z_1\times Y_1)\oplus \CH_d(Y_1\times D_1\oplus\CH_d((Y_1\cap Y_2)\times Y_1)\\\to
\CH_d(Z_1\times Y_1\cup    Y_1\times D_1\cup (Y_1\cap Y_2)\times Y_1),
\end{multline*}
the relation $N\cdot [\Delta_{Y_1^0}] =0$
in $\CH_d((Y_1^0\setminus Z_0)\times (Y_1^0\setminus D_0))$ lifts to a relation of the desired form in $\CH_d(Y_1\times_k Y_1)$. 
\end{proof}

We conclude this series of specialization results with the following variation on Lemma~\ref{lem:Spec2}; a similar result may be found in \cite[Lemma 2.2]{CTNew}.

\begin{lemma}\label{lem:Spec3} Let $\sO$ be a discrete valuation ring with quotient field $K$ and residue field $k$. Let $f:\sX\to\Spec \sO$ be a flat and proper morphism of dimension $d$ over $\Spec \sO$ with generic fiber $X$ and special fiber $Y$. We suppose $Y$ is a union of closed subschemes, $Y=Y_1\cup Y_2$, with $X$ and  $Y_1$ separable and geometrically irreducible.  Suppose that $X$ admits a decomposition of the diagonal of order $N$. Let $Z=(Y_1\cap Y_2)_{red}$ with inclusion $i_Z:Z\to Y_1$. Suppose further that $Y_{2k(Y_1)}$ admits a zero-cycle $y_2$ of degree r supported in the smooth locus of $Y_{2k(Y_1)}$.

Then there is an identity in $\CH_d(Y_1\times_k Y_1)$
\[
Nr\Delta_{Y_1}=\gamma_1+\gamma_2
\]
with  $\gamma_1$ supported on $Y_1\times D_1$, for some divisor $D_1\subset Y_1$, and $\gamma_2$  supported in $Z\times Y_1$.
\end{lemma}

\begin{proof} 
Let $\eta_1$ be the generic point of $Y_1$, let $\sO_1=\sO_{\sX,\eta_1}$ and let $\sD$ be the henselization of $\sO_1$. Let $L$ be the quotient field of $\sD$; clearly $\sD$ has residue field $k(Y_1)$. Then as $\Spec \sO_1\to \Spec \sO$ is essentially smooth, the base-change $\sX_\sD:=\sX\otimes_\sO\sD\to\Spec\sD$ has generic fiber $\sX_L$ and special fiber $Y_{k(Y_1)}=Y_{1k(Y_1)}\cup Y_{2k(Y_1)}$.  Let $\sX_\sD^{sm}\subset \sX_\sD$ be the maximal  open subscheme of $\sX_\sD$ that is smooth over $\sD$.

Fix a rational equivalence
\[
N\cdot \Delta_X\sim x\times X+\gamma
\]
with $x$ a 0-cycle on $X$ and $\gamma$ supported on $X\times E$ for some divisor $E$. Pulling this back to $X_L$ gives the rational equivalence
\[
N\cdot \Delta_{X_L}\sim x_L\times_L X_L+\gamma_L
\]
with $\gamma_L$   supported on $X_L\times_L E_L$. Let $\sE$ be the closure of $E_L$ in $\sX_\sD$ and let $E_0=\sE\cap Y_{k(Y_1)}$; $E_0$ contains no generic point of $Y_{k(Y_1)}$. Furthermore, since the 0-cycle $y_2$ on $Y_{2k(Y_1)}$ is contained in the smooth locus of $Y_{2k(Y_1)}$, we may find a 0-cycle $y_2'$ on $Y_{2k(Y_1)}$, rationally equivalent to $y_2$, and with support in the smooth locus of $Y_{2k(Y_1)}\setminus(E_0\cup Z_{k(Y_1)})$. Changing notation, we may assume that $y_2$ is supported in the smooth locus of $Y_{2k(Y_1)}\setminus(E_0\cup Z_{k(Y_1)})$.

Since $\sD$ is Hensel, we may lift $\eta_1\in Y_1(k(Y_1))$ to a section $s_1:\Spec \sD\to \sX_\sD$. Since $y_2$ is supported in the smooth locus of $Y_{k(Y_1)}$, we may similarly lift the 0-cycle $y_2$ on $Y_{2k(Y_1)}$ to a cycle $\mathfrak{y}_2$ on $\sX_\sD$ of relative dimension zero and relative degree $r$ over $\sD$.  This gives us the 0-cycle of degree zero $\rho_L:=r\cdot s_1(\Spec L)-\mathfrak{y}_{2L}$ on $X_L$. Since  $\sD$ is local, $\sX_\sD$ is flat over $\sD$ and both $y_2$ and $\eta_1$ are supported in the smooth locus of $Y\setminus E_0$, it follows that both $s_1(\Spec\sD)$ and $\mathfrak{y}_2$ are supported in  $\sX^{sm}_\sD\setminus\sE$, and thus $\rho_L$ is supported in the smooth locus of $X_L\setminus E$.  

Let $p$ be a closed point in the smooth locus of $X_L$, inducing the inclusion $i_p:X_L\times_Lp\to X_L\times_LX_L$. Since $i_p$ is a regular codimension $d=\dim X$ embedding, we have the pull-back map (see  \cite[Chap. 6]{Fulton})
\[
i_p^*:\CH_d(X_L\times_LX_L)\to \CH_0(X_L\times_Lp)
\]
If $\mathfrak{z}$ is a 0-cycle  supported in the smooth locus of $X_L$, $\mathfrak{z}=\sum_jn_jp_j$, we have the map
\[
\mathfrak{z}^*:\CH_d(X_L\times_LX_L)\to \CH_0(X_L)
\]
defined as the sum $\sum_jn_jp_{1*}\circ i_{p_j}^*$. If $\gamma$ is a $d$-cycle on $X_L\times_LX_L$ such that each component of $\gamma$ intersects each subvariety $X_L\times p_j$ properly, then $\gamma^*(\mathfrak{z})$ is well-defined and
\[
\mathfrak{z}^*(\gamma)=\gamma^*(\mathfrak{z}).
\]

We apply these comments to the 0-cycle $\rho_L$ and the cycles $N\cdot \Delta_{X_L}$, $x_L\times_L X_L$ and $\gamma_L$. We get the identities in $\CH_0(X_L)$
\begin{align*}
N\cdot \rho_L&=\rho_L^*(N\cdot \Delta_{X_L})\\
&=\rho_L^*(x_L\times_L X_L)+  \rho_L^*(\gamma_L).
\end{align*}
Both terms in this last line are zero, the first since, as $X_L$ is irreducible, we have $\rho_L^*(x_L\times_L X_L)=\deg(\rho_L)\cdot x_L=0$, and the second since $X_L\times \supp(\rho_L)\cap \supp(\gamma_L)=\0$. In other words, $N\cdot \rho_L=0$ in $\CH_0(X_L)$.

We apply the specialization map
\[
sp:\CH_0(X_L)\to \CH_0(Y_{k(Y_1)})
\]
and find that $N(r\cdot\eta_1-y_2)=0$ in $\CH_0(Y_{k(Y_1)})$. Thus $Nr\cdot\eta_1=0$ in 
$\CH_0(Y_{1k(Y_1)}\setminus Z_{k(Y_1)})$, and by using the localization sequence for the inclusion $Z_{k(Y_1)}\to Y_{k(Y_1)}$, there is a 0-cycle $\gamma_{2k(Y_1)}$ on $Z_{k(Y_1)}$ with
\[
Nr\cdot\eta_1=i_{Z*}(\gamma_{2k(Y_1)})
\]
in $\CH_0(Y_{1k(Y_1)})$. Spreading this relation out over $Y_1$ as in previous proofs gives the desired decomposition of $Nr\cdot\Delta_{Y_1}$.
\end{proof}

\begin{remark}\label{rem:Spec3} Suppose we have $\sX$, $Y=Y_1\cup Y_2$ and $Z=Y_1\cap Y_2$ satisfying the hypotheses of Lemma~\ref{lem:Spec3}; suppose in addition that $Y_1$ is smooth over $k$. Then for all fields $F\supset k$, the quotient group $\CH_0(Y_{1F})/i_{Z*}(\CH_0(Z_F))$ is $Nr$-torsion. Indeed, since $Y_1$ is smooth, we have an operation of correspondences on $\CH_0(Y_{1F})$, the correspondence $\gamma_1^*$ of Lemma~\ref{lem:Spec3} acts trivially on $\CH_0(Y_{1F})$, $\gamma_2^*$  maps $\CH_0(Y_{1F})$ to $i_{Z*}(\CH_0(Z_F))$ and the sum acts by multiplication by $Nr$.
\end{remark}

The torsion orders behave well with respect to base-change.
 
\begin{lemma}\label{lem:Index} Let $X$ and $Y$ be  proper separable $k$-schemes, with $Y$ integral and with $X$ equi-dimensional over $k$. Let $K$ be the function field $k(Y)$, $I_Y$ the index  of $Y$. \\
1. For all $i$,  $\Tor_k^{(i)}(X)$ is finite if and only if $ \Tor_K^{(i)}(X_K)$ is finite and in this case, $\Tor_k^{(i)}(X)$ divides $I_Y\Tor_K^{(i)}(X_K)$. \\
 2.  Suppose $X$ is geometrically integral.  If  $g\Tor_k(X)$ is finite, then so is $\Tor_k(X)$ and  $\Tor_k(X)$ divides $I_X\cdot g\Tor_k(X)$.\\
 3. Let  $k\subset L$ be an extension of fields with $k$ algebraically closed. Then  $\Tor_k^{(i)}(X)=\Tor_L^{(i)}(X_L)$  for all $i$. Suppose in addition $X$ is smooth and integral. Then $g\Tor_k(X)=g\Tor_L(X_L)$ and $\Tor_k(X)=g\Tor_k(X)$.
 \end{lemma}
 
 \begin{proof}(1)  If $\Tor^{(i)}_k(X)$ is finite, then so is $\Tor^{(i)}_K(X_K)$ by Lemma~\ref{lem:Basic}(4). Suppose $\Tor^{(i)}_K(X_K)$ is finite. Let $y$ be a closed point of $Y$, contained in the smooth locus of $Y$ over $k$, and let $\sO:=\sO_{Y,y}$. Applying Lemma~\ref{lem:specialization} to the constant family $\sX:=X\times_k\sO$, we see that $\Tor^{(i)}_{k(y)}(X_{k(y)})$ is finite and $\Tor^{(i)}_{k(y)}(X_{k(y)})$ divides $\Tor^{(i)}_K(X_K)$. Applying Lemma~\ref{lem:Basic}(4) again, $\Tor^{(i)}_k(X)$ is finite and divides $[k(y):k]\cdot \Tor^{(i)}_{k(y)}(X_{k(y)})$. This proves the first assertion.

For (2), let $y$ be a closed point of $X$, contained in the smooth locus of $X$ over $k$, let $\sO:=\sO_{X,y}$, and let $\eta\in X(k(X))$ be the canonical point, that is, the restriction of the diagonal section $X\to X\times_kX$ to $\Spec k(X)$.  As in  the proof of Lemma~\ref{lem:specialization}, we may stratify $\Spec\sO$ by regular closed subschemes $y=Z_0\subset \ldots\subset Z_d=\Spec \sO$, $d=\dim_kX$, and thereby define specialization homomorphisms
\[
sp_i:\CH_0(X_{k(Z_i)(X)})\to \CH_0(X_{k(Z_{i-1})(X)}); \quad i=1,\ldots, d.
\]
Letting $sp_y:\CH_0(X_{k(X\times_kX)})\to \CH_0(X_{k(y)(X)})$ be the composition of the $sp_i$, we have   $sp_y(\eta_1-\eta_2)=\eta_y-y_{gen}$, where $\eta_y\in X(k(y)(X))$ is base-change of $y\in X(k(y))$  and $y_{gen}\in X(k(y)(X))$ is the base-change of $\eta\in X(k(X))$. Thus $g\Tor_k(X)\cdot (\eta_y-y_{gen})=0$  in $\CH_0(X_{k(y)(X)})$; pushing forward to $\CH_0(X_{k(X)})$ gives $[k(y):k]\cdot g\Tor_k(X)\cdot \eta-g\Tor_k(X)\cdot y\times_kk(X)=0$ in $\CH_0(X_{k(X)})$. Applying localization gives us the decomposition of the diagonal $\Delta_X$ of order $[k(y):k]\cdot g\Tor_k(X)$; doing this for each closed point $y$ gives us the decomposition of the diagonal of order $I_X\cdot g\Tor_k(X)$, hence $\Tor_k(X)$ is finite  and divides $I_X\cdot g\Tor_k(X)$.

 For (3), we may assume that $L$ is finitely generated over $k$, so that $L=k(Y)$ for some integral proper  $k$-scheme $Y$. Since $k$ is algebraically closed, $I_Y=1$, so the first assertion for $\Tor^{(i)}$ follows from (1). The assertions about $g\Tor$ follow from this, (2) and Lemma~\ref{lem:Basic}. 
 \end{proof}
 For example, $\Tor^{(i)}_k(X)=\Tor^{(i)}_L(X_L)$ if $L$ is a pure transcendental extension of a field $k$. 

\begin{definition} Let $X$ be a  proper, separable $k$-scheme. Let $\bar{k}$ be the algebraic closure of $k$ and define $\Tor^{(i)}(X):=\Tor_{\bar k}^{(i)}(X_{\bar k})$. We call $\Tor^{(i)}(X)$ the $i$th {\em geometric torsion order} of $X$. We write $\Tor(X)$ for $\Tor^{(0)}(X)$.
\end{definition}

Note that  $\Tor^{(i)}(X)$  is invariant under base-extension $X\leadsto X_L$ for a field extension $L\supset k$. Also,  assuming $X$ to be smooth and geometrically integral, $\Tor(X)$ is equal to $g\Tor_{\bar k}(X_{\bar k})$.

In much the same vein as Lemma~\ref{lem:Basic}, we show that the generic torsion order measures the torsion order after adjoining a ``generic'' rational point, that is:

\begin{lemma} \label{lem:Basic2} Let $X$ be a smooth proper geometrically integral $k$-scheme and let $K=k(X)$. Then $g\Tor_k(X)=\Tor_K(X_K)$.
\end{lemma}

\begin{proof} If $N\cdot(\eta_1-\eta_2)=0$ in $\CH_0(X_{k(X\times_kX)})$, then we have a decomposition of the diagonal of order $N$ for $X_{k(X)}$:
\[
N\cdot \Delta_{X_K}= N\cdot\eta\times_KX_K +\gamma
\]
with $\gamma$ supported in $X_K\times_K D$, with $D\subsetneq X_K$, and with $\eta$ the restriction of $\Delta_X$ to $X\times_kk(X)\subset X\times_kX$. Thus $\Tor_K(X_K)$ divides $g\Tor_kX$. Conversely, if $X_K$ admits a decomposition of the diagonal of order $n$,
\[
n\cdot \Delta_{X_K} =x\times X_K+\gamma
\]
with $x$ a 0-cycle on $X_K$ and $\gamma$ supported on $X_K\times D$ for some divisor $D\subset X_K$, then applying $n\cdot\Delta_{X_K}^*$ to $\eta$ gives us $x=n\cdot\eta$ in $\CH_0(X_K)$, so $n\cdot \Delta_{X_K} =n\cdot \eta\times X_K+\gamma$ in $\CH_d(X_K\times_KX_K)$. Restriction to $X\times_KK(X_K)$ gives $n\cdot \eta_1=n\cdot\eta_2$ in $\CH_0(X_{k(X\times_kX)})$, so $g\Tor_k(X)$ divides $\Tor_K(X_K)$.
\end{proof}

One last elementary property of the torsion indices concerns the behavior with respect to morphisms

\begin{lemma} \label{lem:TorDegree} Let $f:Y\to X$ be a surjective morphism of  integral reduced proper $k$-schemes of the same dimension $d$. Then $\Tor^{(i)}_k X$ divides $\deg f\cdot \Tor^{(i)}_k Y$ for all $i$. If $X$ and $Y$ are  separable over $k$, then $g\Tor_k X$ divides $(\deg f)^2\cdot g\Tor_k Y$.
\end{lemma}

\begin{proof} Suppose the diagonal for $Y$ admits a decomposition of order $N$ and level $i$:
\[
N\cdot\Delta_Y=\gamma_i+\gamma'
\]
with $\gamma'$ supported on $Y\times D$ for some divisor $D$ and $\gamma_i$ supported on $Z\times Y$ for some closed subset $Z$ of $Y$ with $\dim_kZ\le i$. 
Pushing forward by $f\times f$ gives
\[
\deg f\cdot N\cdot\Delta_X= (f\times f)_*\gamma_i+(f\times f)_*\gamma',
\]
and thus  $\Tor^{(i)}_k X$ divides $\Deg f\cdot \Tor^{(i)}_k Y$.
Similarly, we have $(f\times f\times f)_*(\Delta_{Y, ij})=(\Deg f)^2\cdot \Delta_{X, ij}$ for $ij=12, 13$, which shows that $g\Tor_k X$ divides $(\Deg f)^2\cdot g\Tor_k Y$.
\end{proof}

The behavior of the torsion indices with respect to rational and birational maps will be discussed in the next  section. 

\section{Universally and totally $\CH_0$-trivial morphisms}\label{sec:trivial}
We  recall the notion  of a universally $\CH_0$-trivial morphism and a related notion, that of a totally $\CH_0$-trivial morphism.

\begin{definition}[\hbox{\cite[Definitions 1.1, 1.2]{CTP}}]
Let $p:Z\to Y$ be a proper morphism of finite type $k$-schemes for some field $k$. The morphism  $p$ is {\em universally $\CH_0$-trivial} if for all field extensions $F\supset k$, the map $p_*:\CH_0(Z_F)\to \CH_0(Y_F)$ is an isomorphism. A proper $k$-scheme $\pi_Y:Y\to \Spec k$  is called a universally $\CH_0$-trivial $k$-scheme if $\pi_Y$ is a universally $\CH_0$-trivial morphism.
\end{definition}

\begin{definition}
A proper morphism  $p:Z\to Y$  of  $k$-schemes is {\em totally  $\CH_0$-trivial} if for each point $y\in Y$, the  fiber $p^{-1}(y)$ is a universally $\CH_0$-trivial $k(y)$-scheme. 
\end{definition}
It follows directly from the definition that the property of a proper morphism being totally  $\CH_0$-trivial is stable under arbitrary base-change.

We rephrase a result of Colliot-Th\'el\`ene and Pirutka.

\begin{proposition}[\hbox{\cite[Proposition 1.7]{CTP}}]\label{prop:Tot} Let $p:Z\to Y$ be a totally $\CH_0$-trivial morphism. Then  $p$ is  universally $\CH_0$-trivial. 
\end{proposition}

\begin{remarks} \label{rem:Tot} 1. By the base-change property of totally  $\CH_0$-trivial morphisms, we see that for  $p:Z\to Y$  a totally $\CH_0$-trivial morphism and $W\to Y$ a morphism of   $k$-schemes, the projection $Z\times_YW\to W$ is universally $\CH_0$-trivial.\\
2. There are examples of universally $\CH_0$-trivial morphisms that are not totally $\CH_0$-trivial\footnote{For example, let $k$ be an algebraically closed field of characteristic $\neq2$, let $S$ be the cone in $\P^3_k$ over a smooth plane curve $C$ of degree $\ge3$,  let $Y\to S$ be the double cover branched over the transverse intersection of $S$ with a quadric, and let $y_1, y_2\in Y$ be the points lying over the vertex of $S$.   Let $p:Z\to Y$ be the blow-up of $Y$ at $y_1$ and let $z=p^{-1}(y_2)$. Then for all fields $L\supset k$, $\CH_0(z_L)\xrightarrow{i_{z*}} \CH_0(Z_L)$  and  $\CH_0(y_{2L})\xrightarrow{i_{y_2*}} \CH_0(Y_L)$ are isomorphisms, and thus $p$ is universally $\CH_0$-trivial. However, $p^{-1}(y_1)\cong C$, so $p$ is not totally $\CH_0$-trivial.}; in particular, the property of a morphism being universally $\CH_0$-trivial is not stable under base-change. 
\end{remarks}

\begin{corollary} \label{cor:UnivTot} 
1. Universally $\CH_0$-trivial morphisms and totally $\CH_0$-trivial morphisms are closed under composition.\\
2. Let $p:Z\to Y$ be a morphism of smooth $k$-schemes that is a sequence of blow-ups with smooth centers. Then $p$ is a totally $\CH_0$-trivial morphism.  \\
3. Suppose that the field $k$ admits resolution of singularities of birational morphisms for smooth $k$-schemes of dimension $\le d$, that is: if $p:Z\to Y$ is a proper birational morphism of smooth $k$-schemes of dimension $\le d$, there is a sequence of blow-ups of $Y$ with smooth centers, $q:W\to Y$, such that resulting birational map $r:W\to Z$ is a morphism.  Then each proper birational morphism $p:Z\to Y$  of smooth $k$-schemes of dimension $\le d$ is totally $\CH_0$-trivial. In particular, this holds for $k$ of characteristic zero, or for $d\le 3$ and $k$ algebraically closed (see \cite{Ab}).
\end{corollary}

\begin{proof} (1) for universally $\CH_0$-trivial morphisms is obvious from the definition and for totally $\CH_0$-trivial morphisms this follows with the help of Proposition~\ref{prop:Tot}.

For (2), we use (1) to reduce to  checking for the blow-up of $Y$ along a smooth closed subscheme $F$, for which the assertion is clear.

For (3), let $y$ be a point of $Y$ and $L\supset k(y)$ a field extension. Dominating $Z$ by a $q:W\to Y$ as above,   we have the maps
\[
\CH_0(q^{-1}(y)_L)\xrightarrow{r_*}\CH_0(p^{-1}(y)_L) \xrightarrow{p_*}\CH_0(\Spec L)=\Z
\]
which, as $\CH_0(q^{-1}(y)_L)\to \CH_0(\Spec L)$ is an isomorphism, gives us a splitting to $p_*$. Applying resolution of singularities to $r:W\to Z$ gives a sequence of blow-ups with smooth centers $s:X\to Z$ such that $t:=r^{-1}s:X\to W$ is a morphism. Since $X\to Z$ is totally $\CH_0$-trivial, the sequence
\[
\CH_0(t^{-1}(q^{-1}(y))_L)\xrightarrow{t_*}\CH_0(q^{-1}(y)_L)\xrightarrow{r_*}\CH_0(p^{-1}(y)_L)  
\]
gives a splitting to $r_*$, so $p_*$ is an isomorphism.
\end{proof}

\begin{lemma}\label{lem:LevelBirat}
1. Let $q:Z\to Y$ be a birational totally $\CH_0$-trivial morphism of integral, separable, $k$-schemes. Let  $N>0$ be an integer, let $Y_i, W, D\subset Y$ be  proper closed subsets with $\dim\, Y_i\le i$, and suppose we have a decomposition of $\Delta_Y$ as
\[
N\cdot \Delta_Y=\gamma+\gamma_1+\gamma_2,
\]
with $\gamma$ supported on $Y_i\times Y$,  $\gamma_1$ supported on $Y\times D$ and $\gamma_2$ supported on $W\times Y$. Then there are proper closed subsets $Z_i, D'\subset Z$ with $\dim\, Z_i\le i$ and a decomposition of $\Delta_Z$ as
\[
N\cdot \Delta_Z=\gamma'+\gamma'_1+\gamma'_2,
\]
with $\gamma'$ supported on $Z_i\times Z$,  $\gamma'_1$ supported on $Z\times D'$ and $\gamma'_2$ supported on $q^{-1}(W)\times Z$.\\
2. Let $q:Z\to Y$ be a birational totally $\CH_0$-trivial morphism of integral, separable, proper $k$-schemes.
Then $\Tor_k^{(i)}(Z)=\Tor_k^{(i)}(Y)$ for all $i$. \\
3. Let  $q:Z\to Y$ be a birational universally $\CH_0$-trivial morphism of  integral proper $k$-schemes.   Then $\Tor_k(Z)=\Tor_k(Y)$. If moreover $Z$ and $Y$ are geometrically integral, then $g\Tor_k(Z)=g\Tor_k(Y)$.\\
\end{lemma}

\begin{proof} We note that (2) follows easily from (1). Indeed, (1) with $W=\0$ shows that $\Tor_k^{(i)}(Z)$ divides $\Tor_k^{(i)}(Y)$ for all $i$; as $(q\times q)_*(\Delta_Z)=\Delta_Y$, it follows that a decomposition of $\Delta_Z$ of order $N$ and level $i$ gives a similar decomposition of $\Delta_Y$ by applying $(q\times q)_*$. 

We now prove (1). We may assume that $W=\0$. Indeed, if we replace $Y$ with $Y':=Y\setminus W$ and $Z$ with $Z':=Z\setminus q^{-1}(W)$, the result for $q_{|Z'}:Z'\to Y'$ and the decomposition 
\[
N\cdot \Delta_{Y'}=\gamma_{|Y'\times Y'}+\gamma_{1|Y'\times Y'},
\]
together with localization gives (1) for the original data. 

Suppose then we have  
\[
N\cdot \Delta_Y= \gamma+\gamma_1
\]
with $\gamma$ supported on $Y_i\times Y$ and $\gamma_1$ supported on $Y\times D$.  Let $K=k(Y)$ and let $\eta_Y\in Y$ be the generic point. We have a rational equivalence of 0-cycles on $Y\times\eta_Y$
\[
N\cdot\eta_Y\times\eta_Y\sim \gamma_{\eta_Y}
\]
with $ \gamma_{\eta_Y}$ a 0-cycle supported on $Y_i\times\eta_Y$.  Thus $N\cdot\eta_Y\times\eta_Y\sim0$ on $(Y\setminus Y_i)\times\eta_Y$.

 Since $Z\setminus q^{-1}(Y_i)\to Y\setminus Y_i$ is birational and universally $\CH_0$-trivial (Remark~\ref{rem:Tot}), there is a rational equivalence of 0-cycles
\[
N\cdot\eta_Z\times\eta_Z\sim0
\]
on $(Z\setminus q^{-1}(Y_i))\times\eta_Z$, where $\eta_Z\in Z$ is the generic point. We claim that there is a  dimension $\le i$ closed subset $Z'$ of $Z$ and a rational equivalence of 0-cycles on $Z\times\eta_Z$
\[
N\cdot\eta_Z\times\eta_Z\sim  \rho_Z
\]
with $\rho_Z$ a 0-cycle supported on $Z'\times\eta_Z$. We proceed by a noetherian induction: We assume there is a closed subset $Y^j\subset Y_i$, a  dimension $\le i$ closed subset $Z_j$ of $q^{-1}(Y_i)$ and a rational equivalence of 0-cycles on $(Z\setminus q^{-1}(Y^j))\times\eta_Z$
\[
N\cdot\eta_Z\times\eta_Z\sim  \rho_j
\]
with $\rho_j$ a 0-cycle supported on $Z_j\times\eta_Z$, and we show the parallel statement for a proper closed subset $Y^{j+1}$ of $Y^j$. The induction starts with $Y^0=Y_i$. 

Chose an irreducible component $Y^j_0$ of $Y^j$ and let $\nu$ be its generic point. Let $Y'$ be the union of the components of $Y^j$ different from $Y^j_0$. We have the exact  localization sequence
\begin{multline*}
\CH_0((q^{-1}(Y^j_0\setminus Y'))\times\eta_Z)\xrightarrow{i_*} \CH_0((Z\setminus q^{-1}(Y'))\times\eta_Z)\\\to 
\CH_0((Z\setminus q^{-1}(Y^j))\times\eta_Z)\to 0
\end{multline*}
and thus there is a 0-cycle $\rho'$ on $q^{-1}(Y^j_0\setminus Y')\times\eta_Z$ and a rational equivalence
\[
N\cdot\eta_Z\times\eta_Z\sim  \rho_j+i_*(\rho')
\]
on $(Z\setminus q^{-1}(Y'))\times\eta_Z$. 

Write 
\[
\rho'=\sum_im_ix_i+\sum_jn_jx_j',
\]
where the $x_i$, $x_j'$ are closed points of $q^{-1}(Y^j_0\setminus Y')\times\eta_Z$, such that $q\circ p_1(x_i)=\nu$ for all $i$ and $q\circ p_1(x_j')$ is contained in some proper closed subset (say $Y''$)  of $Y^j_0$ for all $j$. Replacing $Y'$ with $Y'\cup Y''$ and changing notation, we may assume that $\rho'=\sum_im_ix_i$.

By assumption, the map $q^{-1}(\nu)\to \nu$ is universally $\CH_0$-trivial, so there is a degree one 0-cycle $\epsilon$ on $q^{-1}(\nu)$ so that $\epsilon_L$ generates $\CH_0(q^{-1}(\nu)_L)$ for all field extensions $L\supset k(\nu)$, in particular, $\epsilon\times\eta_Z$ generates $\CH_0(q^{-1}(\nu)\times\eta_Z)$. Enlarging $Y'$ again by a proper closed subset of $Y^j_0$, we may assume that 
\[
\rho'=m\cdot \epsilon\times\eta_Z
\]
in $\CH_0(q^{-1}(Y^j_0\setminus Y')\times\eta_Z)$, for some $m\in\Z$. Since $\epsilon$ is a 0-cycle on $q^{-1}(\nu)$, the closure $Z'$ of the support of $\epsilon$ in $q^{-1}(Y^j_0)$ has dimension over $k$ bounded by the transcendence dimension of $k(\nu)$ over $k$, that is, by $\dim_kY^j_0$; since $Y^j_0\subset Y_i$, we have
\[
\dim_kZ'\le i.
\]
Taking $Y^{j+1}=Y'$, $Z_{j+1}=Z_j\cup Z'$, $\rho_{j+1}=\rho_j+m\cdot \epsilon\times\eta_Z$,  the 0-cycle $\rho_{j+1}$ is supported on $Z_{j+1} \times\eta_Z$, $\dim_k Z_{j+1}\le i$, and we have
\[
N\cdot\eta_Z\times\eta_Z= \rho_{j+1}
\]
in $\CH_0((Z\setminus q^{-1}(Y^{j+1}))\times\eta_Z)$. The induction thus goes through, proving the result.

The proof of (3) is similar but easier. We have already seen that if $Z$ has a decomposition of the diagonal of order $N$, then so does $Y$. If conversely $Y$ has a decomposition of the diagonal of order $N$, then there is a 0-cycle $y$ on $Y$ with 
\[
N\cdot \eta_Y\times\eta_Y=y\times\eta_Y
\]
in $\CH_0(Y\times\eta_Y)$. As $q:Z\to Y$ is universally $\CH_0$ trivial, there is a 0-cycle $z$ on $Z$ with $q_*z=y$ in $\CH_0(Y)$ and since $(q\times q)_*:\CH_0(Z\times\eta_Z)\to \CH_0(Y\times\eta_Y)$ is an isomorphism, we have
\[
N\cdot \eta_Z\times\eta_Z=z\times\eta_Z
\]
in $\CH_0(Z\times\eta_Z)$. The proof for $g\Tor$ is the same.
\end{proof}

We note some consequences of Lemma~\ref{lem:LevelBirat}.

\begin{proposition} \label{prop:Ratl} Let $f:Y\to X$ be a dominant rational map of smooth integral proper $k$-schemes of the same dimension $d$. \\
1. Suppose $k$  admits resolution of singularities for  rational maps of varieties of dimension $\le d$, that is, if  $p:Y\to X$ is a  rational morphism of smooth $k$-schemes of dimension $\le d$, there is a sequence of blow-ups of $Y$ with smooth center, $q:W\to Y$, such that resulting rational map $r:W\to X$ is a morphism. Then $\Tor_k^{(i)} X$ divides $\deg f\cdot \Tor_k^{(i)} Y$ for all $i$.\\
2.  Without assumption on $k$, $\Tor_k X$ divides $\deg f\cdot \Tor_k Y$ and  $g\Tor_k X$ divides $(\deg f)^2\cdot g\Tor_k Y$. 
\end{proposition}

\begin{proof}
For (1)  we may find a   sequence of blow-ups  with smooth centers, $g:Z\to Y$, so that the induced rational map $h:Z\to X$ is a morphism. Since $g$ is a totally $\CH_0$-trivial morphism, $\Tor_k^{(i)} Z=\Tor_k^{(i)}Y$ by Lemma~\ref{lem:LevelBirat}(2), so we may assume that $g$ is a morphism; the result then follows from Lemma~\ref{lem:TorDegree}. 

For (2), let $Z\subset Y\times X$ be the graph of $f$, that is, the closure of the graph of $f:V\xr{} X$ for a non-empty open subset $V\subset Y$ on which $f$ is defined. The map $p_1:Z\xr{} Y$ is birational and there is a non-empty open $X_0\subset X$ such that $p_1:p_2^{-1}(X_0)\cap Z\xr{} Y$ is an open immersion; set $Y_0:=p_1(p_2^{-1}(X_0)\cap Z)$. The correspondence $Z\times_k Z$ yields a homomorphism 
$$
g:\CH_d(Y\times Y) \xr{} \CH_d(X\times X). 
$$   
We claim that $g(\Delta_Y)=\deg(f)\cdot \Delta_X+\gamma$ where $\gamma$ is a cycle supported on $X\times (X\backslash X_0)$, which implies the assertion for $\Tor_k X$. Keeping track of supports and using localization, we have an identity in $\CH_d(Z\times_kZ)$ of the form
\begin{equation}\label{equation-never-ending-cold}
[Z\times Z]\cdot (p_1\times p_1)^*(\Delta_Y)= \Delta_Z + \gamma', 
\end{equation}
where $\gamma'$ has support in $(p_1^{-1}({Y\backslash Y_0})\cap Z)\times_{Y\backslash Y_0}(p_1^{-1}({Y\backslash Y_0})\cap Z)$. Thus $(p_2\times p_2)_*(\gamma')$ has support in $X\times (X\backslash X_0)$. Applying $(p_2\times p_2)_*$ to \eqref{equation-never-ending-cold} we prove our claim.

The proof for $g\Tor_k$ is similar.
\end{proof} 
In particular, if we have resolution of singularities of birational maps, $\Tor_k^{(i)}$ is a birational invariant and in general $\Tor_k$ is a birational invariant; from this it follows easily that $\Tor_k^{(i)}$ is a stable birational invariant  if we have resolution of singularities of birational maps and in general $\Tor_k$ is a stable birational invariant.

\section{Specialization and degeneration}\label{sec:degen}

The  next result, in a somewhat different form,  is proven in  \cite[Th\'eor\`eme 1.12]{CTP}. In a less general setting, a similar result may be found in \cite[Theorem 1.1]{Voisin}.

\begin{proposition}  \label{prop:specialization} Let $\sO$ be a regular local ring with quotient field $K$ and residue field $k$. Let $f:\sX\to\Spec \sO$ be a flat and proper morphism with geometrically integral fibers and let $X$ be the generic fiber $\sX_K$, $Y$ the special fiber $\sX_k$. We suppose that $Y$ admit a resolution of singularities $q:Z\to Y$ such that $q$ is a universally $\CH_0$-trivial morphism. Suppose in addition that $X$ admits a decomposition of the diagonal of order $N$.   Then $Z$ also admits a decomposition of the diagonal of order $N$. In particular, if $\Tor_K(X)$ is finite then so is
$\Tor_k(Z)$, and in this case
$\Tor_k(Z)\ |\ \Tor_K(X)$. 
\end{proposition}
In \cite{CTP} it is assumed that $X$ has a resolution of singularities $\tilde{X}\to X$ such that $\tilde{X}_K$ admits a decomposition of the diagonal of order $N$, which implies the same condition on $X$ by pushing forward; there is also an assumption that $Z$ has a 0-cycle of degree 1. This resolution of singularities in \cite{CTP} arises because they consider decompositions of the diagonal only on smooth proper varieties; the existence of a degree 1 0-cycle comes from considering only the case $N=1$. The modified version stated above is proved exactly as as in {\it loc cit.}

We prove an extension of this specialization result which takes the decompositions of higher level into account.

\begin{proposition}  \label{prop:specialization2}  Let $\sO$ be a regular local ring with quotient field $K$ and residue field $k$. Let $f:\sX\to\Spec \sO$ be a flat and proper morphism with geometrically integral fibers and let $X$ be the generic fiber $\sX_K$, $Y$ the special fiber $\sX_k$.  Suppose that there is a  birational totally $\CH_0$-trivial morphism $q:Z\to Y$  of geometrically integral proper $k$-schemes.  \\
1. Suppose  $X$ admits a decomposition of the diagonal of order $N$ and level $i$.  Then $Z$ also admits a decomposition of the diagonal of order $N$ and level $i$.  If $\Tor_K^{(i)}(X)$ is finite then so is $\Tor^{(i)}_k(Z)$ and  in this case $\Tor^{(i)}_k(Z)\ |\ \Tor^{(i)}_K(X)$.\\
2. Let $\bar K$ and $\bar k$ be the respective algebraic closures of $K$ and $k$ and suppose that $X_{\bar K}$ admits a decomposition of the diagonal of order $N$ and level $i$. Suppose that $K$ has characteristic zero, or that $\sO$ is excellent. Then $Z_{\bar k}$ also admits a decomposition of the diagonal of order $N$ and level $i$.  If $\Tor^{(i)}(X)$ is finite then so is $\Tor^{(i)}(Z)$ and  in this case
$\Tor^{(i)}(Z)\ |\ \Tor^{(i)}(X)$.
\end{proposition}

\begin{proof}
The assertion (2) follows from (1) by first stratifying $\Spec \sO$ as in the proof of Lemma~\ref{lem:specialization} to reduce to the case of a DVR. We then take a finite extension $L$ of $K$ so that $\Tor^{(i)}(X)=\Tor^{(i)}_L(X_L)$, take the normalization $\sO\to \sO^N$ of $\sO$ in $L$ and replace $\sO$ with the localization $\sO'$ of $\sO^N$ at some maximal ideal. Letting $k'$ be the residue field of $\sO'$, $\Tor^{(i)}(Z)$ divides $\Tor^{(i)}_{k'}(Z_{k'})$, so (1) implies (2).  We now prove (1).

By  Lemma~\ref{lem:specialization}, $Y$ admits a  decomposition of the diagonal of order $N$ and level $i$. By Lemma~\ref{lem:LevelBirat}, $Z$ also admits a a decomposition of the diagonal of order $N$ and level $i$, proving (1).
\end{proof}

We also have a version that incorporates Totaro's extended specialization Lemma~\ref{lem:Spec2}.

\begin{proposition}  \label{prop:specialization3}  Let $\sO$ be a discrete valuation ring with quotient field $K$ and residue field $k$. Let $f:\sX\to\Spec \sO$ be a flat and proper morphism of dimension $d$ over $\Spec \sO$ with generic fiber $X$ and special fiber $Y$. We suppose $Y$ is a union of closed subschemes, $Y=Y_1\cup Y_2$  and  that $X$ and $Y_1$ are geometrically integral. Suppose there is a  birational totally $\CH_0$-trivial morphism $q:Z\to Y_1$  of geometrically integral proper $k$-schemes and that $X$ admits a decomposition of the diagonal of order $N$ and level $i$. Then there are proper closed subsets $Z_i, D\subset Z$ with $\dim Z_i\le i$ and a decomposition
\[
N\cdot\Delta_Z=\gamma+\gamma_1+\gamma_2
\]
with $\gamma$ supported in $Z_i\times Z$, $\gamma_1$ supported in $Z\times D$ and $\gamma_2$ supported in $q^{-1}(Y_1\cap Y_2)\times Z$.
\end{proposition}

\begin{proof}
This follows directly from Lemma~\ref{lem:Spec2} and Lemma~\ref{lem:LevelBirat}.
\end{proof}

\begin{remark} \label{remark-geometric-specialization} As in the second part of Proposition~\ref{prop:specialization2}, we may take the $N$ in Proposition~\ref{prop:specialization3} to be $\Tor^{(i)}_{\bar{K}}(X_{\bar{K}})$ if $\sO$ is excellent or if $K$ has characteristic zero, by replacing $\sO$ with its normalization $\sO'$ in a  finite extension $L$ of $K$ so that $\Tor^{(i)}_{\bar{K}}(X_{\bar{K}})=\Tor^{(i)}_L(X_L)$, replacing $\sX$ with $\sX\times_\sO\sO'$, replacing $k$ with the residue field $k'$ of $\sO'$ and replacing $Z$ with $Z\otimes_kk'$.
\end{remark}

\section{Torsion order for complete intersections in a projective space: an upper bound}\label{sec:Roitman} We concentrate on the 0th  torsion order of a (reduced, separable) complete intersection $X=X^n_{d_1,\ldots, d_r}$ in $\P^{n+r}$ of dimension $n$ and multi-degree $d_1, d_2, \ldots, d_r$. In this section, we recall the construction of Roitman \cite{Roitman},  which gives an upper bound for $\Tor_k(X)$; by Lemma~\ref{lem:Basic}(1), this gives an upper bound for $\Tor^{(i)}_k(X)$ for all $i$.

We often shorten the notation by writing $d_*$ for a sequence $d_1, d_2, \ldots, d_r$.
 
\begin{proposition}\label{prop:UpperBound} Let $k$ be a field and let   $X=X^n_{d_1,\ldots, d_r}$ in $\P^{n+r}_k$ with $\sum_id_i\le n+r$ be a reduced, separable complete intersection of multi-degree $d_1,\ldots, d_r$, with $n\ge1$. Then  $\Tor_k(X)$ is finite and divides $\prod_{i=1}^rd_i!$.
\end{proposition}

\begin{proof} The reduced, separable complete intersections in $\P^{n+r}$ and of multi-degree $d_1,\ldots, d_r$ are para\-metrized by an open subscheme $\sU_{d_*;n}$ of a product of projective spaces; by Lemma~\ref{lem:specialization} it suffices to prove the result for the subscheme $X:=X_{d_*, gen}$ of $\P^{n+r}_K$ defined over the field $K:=k(\sU_{d_*;n})$  corresponding to the generic point of $\sU_{d_*;n}$. For such an $X$, there is an open subset $V\subset X$, such that, for $x\in V$, the set of lines $\ell\subset \P^{n+r}$ such that $x\in \ell$ and $(\ell\cap X)_\red$ is either $\{x\}$  or is $\ell$ is  defined by a complete intersection $W_x$ of  multi-degree 
\[
d_1-1, d_1-2,\ldots, 2, 1, d_2-1, d_2-2,\ldots 2,1,\ldots, d_r-1,\ldots, 2,1
\]
in the projective space $\P^{n+r-1}_{K(x)}$ of lines through $x$.
Indeed, we may choose a standard affine open $U$ in $\P^{n+r}_{K(x)}$ containing $x$ and chose affine coordinates $t_0,\ldots, t_{n+r-1}$ for $U$ so that $x$ is the origin, and $X\cap U$ is defined by inhomogeneous equations $F_1=\ldots=F_r=0$. Writing each $F_i$ as a sum of homogeneous terms $F_i^{(j)}$ of degree $j$, 
\[
F_i=\sum_{j=1}^{d_i}F_i^{(j)},
\]
$W_x$ is defined by ideal $(\ldots F_i^{(j)}\ldots)$, $i=1,\ldots, r$, $j=1,\ldots, d_i-1$. Since we are choosing $X$ to be the generic hypersurface, and as we may also chose $x$ to lie outside any proper closed subset of $X$, the homogeneous terms $F_i^{(j)}\in K(x)[t_0,\ldots, t_{n+r-1}]_j$ will define a complete intersection in $\P^{n+r-1}_{K(x)}$. In particular $W_x$  has codimension $\sum_{i=1}^r(d_i-1)\le n+r-1$ in $\P^{n+r-1}_{K(x)}$, is non-empty (Bezout's theorem!) and has degree $\prod_{i=1}^r(d_i-1)!$.  

Let $W^0_x\subset W_x$ be the closed subset of lines $\ell$ containing $x$ with $\ell\subset X$; this is defined by the $r$ additional equations $F_i^{(d_i)}=0$. Thus, for general $(X, x)$, $W^0_x$ has codimension $r$ on $W_x$ (or is empty).

Since $n+r-1-\sum_{i=1}^r(d_i-1)\ge r-1$, we may intersect $W_x$ with a suitably general linear space $L\subset \P^{n+r-1}_{K(x)}$ to form a closed subscheme $\bar{W}_x\subset W_x$ of dimension $r-1$ and degree $\prod_{i=1}^r(d_i-1)!$ and we may chose $L$ with $L\cap W^0_x=\0$. The cone over $\bar{W}_x$ with vertex $x$, $\sC_x\subset \P^{n+r}_{K(x)}$, is thus a dimension $r$ closed subscheme of degree $\prod_{i=1}^r(d_i-1)!$ with intersection (set) $\sC_x\cap X=\{x\}$. Thus as cycles
\[
\sC_x\cdot X=(\prod_{i=1}^r d_i!)\cdot x.
\]

Let $\eta$ be the generic point of $X$.  Taking $x=\eta$ in the above discussion gives
\[
\prod_{i=1}^rd_i!\cdot\eta =\sC_\eta\cdot X.
\]
But $\sC_\eta$ is an $r$-cycle on $\P^{n+r}_{K(\eta)}$ of degree $\prod_{i=1}^r(d_i-1)!$, so we have $\sC_\eta=\prod_{i=1}^r(d_i-1)!\cdot L_r$ in $\CH_r(\P^{n+r}_{K(\eta)})$, where $L_r\subset \P^{n+r}_K$ is any dimension $r$ linear subspace. Since $K$ is infinite, we may choose $L_r$ so that the intersection $L_r\cap X$ has dimension zero. Thus, letting $z=\prod_{i=1}^r(d_i-1)!\cdot (L_r\cdot X)$, we have
\[
\prod_{i=1}^rd_i!\cdot\eta-z_{K(\eta)}=0
\]
in $\CH_0(X_{K(\eta)})$, which gives a decomposition of the diagonal in $X$ of order $\prod_{i=1}^rd_i!$. Thus $\Tor_K(X)$ is finite and divides $\prod_{i=1}^rd_i!$, as desired.\end{proof}

\begin{corollary}\label{cor:UpperBound} Let $X=X^n_{d_1,\ldots, d_r}$ in $\P^{n+r}_k$ be a smooth complete intersection of multi-degree $d_1,\ldots, d_r$ and of dimension $n\ge1$  with $\sum_id_i\le n+r$. Then  $g\Tor_k(X)$ and  $\Tor(X)$ are both finite and both divide $\prod_{i=1}^rd_i!$.
\end{corollary}

\begin{proof}   Both $g\Tor_k(X)$ and $\Tor(X):=\Tor_{\bar{k}}(X_{\bar{k}})$ divide $\Tor_k(X)$ (Lemma~\ref{lem:Basic}) so the result follows from Proposition~\ref{prop:UpperBound}.
\end{proof}

 \section{The generic case}\label{sec:generic} In this section we discuss the case of the generic complete intersection. Let $k$ denote a fixed base-field, for instance the prime field. The bounds we find for the generic case are independent of $k$, so one could equally well take $k$ to be the reader's favorite field, even an algebraically closed one. 
 
Before going into details, we outline the case of hypersurfaces, which uses all the main ideas. 

Let $d!^*$ denote the l.c.m.\ of the integers $2,\ldots, d$. Note that $d!^*$ is inductively the l.c.m.\ of $d$ and $(d-1)!^*$ (Lemma~\ref{lem:Product}). Our main result in the case of hypersurfaces is that the torsion order of level 0 of the generic hypersurface of degree $d\le n+1$ in $\P^{n+1}$ is divisible by $d!^*$, in other words, if the generic hypersurface admits a decomposition of the diagonal of degree $N$, then $d!^*$ divides $N$. 

The hypersurfaces of degree $d\le n+1$ in $\P^{n+1}_k$ are parametrized by a projective space $\P^{N_{n,d}}$
and it is not hard to show that the index  over $k(\P^{N_{n,d}})$ of the generic degree $d$ hypersurface $X$ is $d$. In fact, we have a much stronger statement,  namely $\CH_0(X)=\Z$, generated by $X\cdot \ell$ for $\ell\subset \P^{n+1}$ a line (Lemma~\ref{lem:ChowGp}(1)). 

 If we have a decomposition of order $N$ of the diagonal on  $X$,
\[
N\cdot\Delta_X\sim x\times X+\gamma,
\]
then since $N=\deg_{k(\P^{N_{n,d}})}x$, it follows that $d|N$. Now degenerate $X$ to the generic degree $d-1$ hypersurface $Y$ in  $\P^{n+1}$ plus the hyperplane $H$ given by $x_{n+1}=0$, and let $Z=Y\cap H$. Here $Y$ and $Z$ are defined over $L:=k(\P^{N_{n,d-1}})$. Specializing the above rational equivalence using Lemma~\ref{lem:Spec2} gives a rational equivalence on $Y\times_L Y$ of the form
\[
N\cdot \Delta_Y\sim \bar{x}\times Y+\gamma_1+\gamma_2
\]
with $\bar{x}$ a zero-cycle on $Y$, $\gamma_1$ a dimension $n$ cycle on $Z\times_LY$ and $\gamma_2$ supported in $Y\times D$ for some divisor $D$ on $Y$.   Passing to the generic point of $Y$, $\gamma_1$ gives a 0-cycle on $Z\times_LL(Y)$. The main point is to show that $\CH_0(Z\times_LL(Y))$ is also $\Z$, generated by intersections from $\P^{n-1}$ (Lemma~\ref{lem:ChowGp}(3)), so we can replace $\gamma_1$ with $y\times Y+\gamma_3$, where $y$ is a 0-cycle on $Z$ and  $\gamma_3$ is supported on $Z\times D'$ for some divisor $D'$ on $Y$ (Lemma~\ref{lem:Decomp}). In other words,
\[
N\cdot \Delta_Y\sim (\bar{x}+y)\times Y+\gamma_2+\gamma_3,
\]
so $Y$ admits a decomposition of the diagonal of degree $N$. 
Now use induction on $d$ to conclude that  $(d-1)!^*|N$. As we already know that $d|N$, we find $d!^*|N$.

Now for the details. Fix integers $n, r\ge1$. For an integer $d$, let  $\sS_{d, n+r}$ be the set of indices $I=(i_0,\ldots, i_{n+r})$ with $0\le i_j$ and $\sum_ji_j=d$.  We let $\sS_i=\sS_{d_i, n+r}$ and let $N_i:=\# S_i$. Let $\{ u_i^{(I)} | I\in \sS_i\}$ be homogeneous coordinates for $\P^{N_i}$ and let $x_0,\ldots, x_{n+r}$ be homogeneous coordinates for $\P^{n+r}$. The universal family of intersections of multi-degree $d_1, \ldots, d_r$ in $\P^{n+r}$, $\sX^{d_*, n}$, is the subscheme of $\P^{N_1}\times\ldots\times\P^{N_r}\times\P^{n+r}$ defined by the multi-homogeneous ideal in the polynomial ring $k[\{u^{(I)}_i\}_{I\in \sS_i, i=1,\ldots, r}, x_0,\ldots, x_{n+r}]$ generated by the elements
 \[
 \sum_{I\in \sS_i}u_i^{(I)} x^I;\quad i=1,\ldots, r
\]
where as usual $x^I=x_0^{i_0}\cdots x_{n+r}^{i_{n+r}}$ for $I=(i_0,\ldots, i_{n+r})$. We let $\eta:=\eta_{d_*;n}$ denote the generic point of  $\P^{N_1}\times\ldots\times\P^{N_r}$ and let $\sX^{d_*, n}_{\eta}$ denote the fiber product 
\[
\sX^{d_*, n}_\eta:=\sX^{d_*, n}\times_{\P^{N_1}\times\ldots\times\P^{N_r}}\eta\subset \P^{n+r}_\eta.
\]

By Proposition~\ref{prop:UpperBound}, we know that if $\sum_{i=1}^rd_i\le n+r$, then $\Tor_{k(\eta)}(\sX^{d_*, n}_\eta)$ is finite and divides $\prod_id_i!$. We turn to a computation of a lower bound.  

 Let $H\subset \P^{N_1}\times\ldots\times\P^{N_r}\times\P^{n+r}$ be the subscheme defined by $(x_{n+r}=0)$, let $\sX^{d_*, n}_H:=\sX^{d_*, n}\cap H$ and let $\sX^{d_*, n}_{H,\eta}:=\sX^{d_*, n}_\eta\cap H$. Let $\eta':=\eta_{d_*, n-1}$.

We separate the indices $\sS_i$ into two disjoint subsets $\sS_i^0$ and $\sS_i^1$, with $\sS_i^0$ the set of $(i_0,\ldots, i_{n+r})$ with $i_{n+r}=0$ and $\sS_i^1$ those with $i_{r+n}>0$. We set $v_i^{(I)}=u_i^{(I)}$ for $I\in \sS_i^0$ and   $w_i^{(I)}=u_i^{(I)}$ for $I\in \sS_i^1$. We write $k(\{u^{(I)}_i\}_0)$ for the field extension of $k$ generated by the ratios $u^{(I)}_i/u^{(I')}_i$ $I\neq I'$, and similarly for $k(\{v^{(I)}_i\}_0)$, giving us the field extension $k(\{v^{(I)}_i\}_0)\subset k(\{u^{(I)}_i\}_0)$.   We note that $k(\{u^{(I)}_i\}_0)=k(\eta)$, 
$k(\{v^{(I)}_i\}_0)=k(\eta')$   and the $k(\eta)$-scheme $\sX^{d_*, n}_{H,\eta}$ is canonically isomorphic to the base-change of the $k(\eta')$-scheme $\sX^{d_*, n-1}_{\eta'}$ via the base extension $k(\eta')\subset k(\eta)$:
\[
\sX^{d_*, n}_{H,\eta}\cong \sX^{d_*, n-1}_{\eta'}\otimes_{k(\eta')}k(\eta).
\]
This defines for us the projection $q_1:\sX^{d_*, n}_{H,\eta}\to \sX^{d_*, n-1}_{\eta'}$.

Let $K=k(\eta)(\sX^{d_*;n}_\eta)=k(\sX^{d_*,n})$. 
We have the morphism of $k(\eta')$-schemes 
\[
\pi:\sX^{d_*, n}_{H,\eta}\otimes_{k(\eta_{d_*;n})}K\to \sX^{d_*, n-1}_{\eta'}
\]
formed by the composition
\[
\sX^{d_*, n}_{H,\eta}\otimes_{k(\eta)}K 
\xrightarrow{p_1}\sX^{d_*, n}_{H,\eta}\xrightarrow{q_1} \sX^{d_*, n-1}_{\eta'}
\]

\begin{lemma}\label{lem:ChowGp} 1. For $i=0,\ldots, n$, the intersection map
\[
\CH_{r+i}(\P^{n+r}_{k(\eta)})\to \CH_i(\sX^{d_*;n}_\eta)
\]
is an isomorphism.\\
2. For $i=0,\ldots, n-1$, the pullback
\[
\pi^*:\CH_i(\sX^{d_*, n-1}_{\eta'})\to \CH_i(\sX^{d_*, n}_{H,\eta}\otimes_{k(\eta)}K)
\]
is an isomorphism. \\
3. For $i=0,\ldots, n-1$, the intersection map
\[
\CH_{r+i}(\P^{n+r}_K)\to \CH_i(\sX^{d_*, n}_{H,\eta}\otimes_{k(\eta)}K)
\]
is an isomorphism.
\end{lemma}

\begin{proof} Noting that the base-extension $\CH_*(\P^{n+r}_{k(\eta)}) \to \CH_*(\P^{n+r}_K)$ is an isomorphism, the assertion (3)  follows from (1) (for $n-1$) and (2). For (1), the projection
\[
p_2:\sX^{d_*, n}\to\P^{n+r}
\]
expresses $\sX^{d_*, n}$ as $\P^{N_1-1}\times\ldots\times \P^{N_r-1}$-bundle over $\P^{n+r}$, with fibers embedded  in $\P^{N_1}\times\ldots\times \P^{N_r}$ linearly in each factor.  Thus $\CH_*(\sX^{d_*;n})$ is generated by  $\CH_*(\P^{N_1}\times\ldots\times \P^{N_r}\times\P^{n+r})$ via restriction. After localization at $\eta$, this shows that $\CH_*(\sX^{d_*;n}_\eta)$ is generated by $\CH_*(\P^{n+r}_{k(\eta)})$ via restriction.  The fact that the surjective map $\CH_{r+i}(\P^{n+r}_{k(\eta)})\to \CH_i(\sX^{d_*;n}_\eta)$ is also injective in the stated range follows by noting that the intersection pairing on $\sX^{d_*;n}_\eta$  is non-degenerate when restricted to these cycles. This proves (1).

For (2), fix for each $i$ the index $I^0_i:=(d_i,0,\ldots,0)$, and the index $I_i^1:=(0,\ldots, 0,d_i)$, and for each homogeneous variable  $w_i^{(I)}$, let $w_i^{(I)0}$ be the corresponding affine coordinate $w_i^{(I)}/v_i^{(I_i^0)}$. Similarly, we let $v_i^{(I)0}=v_i^{(I)}/v_i^{(I_i^0)}$. Let $y_i=x_i/x_0$, $i=1,\ldots, n+r$, $y_0=1$. The field extension $k(\eta')\to K$ is isomorphic to the field extension given by including the constants $k(\{v^{(I)}_i\}_0)$ of the  $k(\{v^{(I)}_i\}_0)$-algebra $A$,
\[
A:=k(\{v^{(I)}_i\}_0, y_1,\ldots, y_{n+r})[\{w_i^{(I)0}\}]/(\ldots, \sum_{I\in\sS_i^0}v_i^{(I)0}\cdot y^I + \sum_{I'\in \sS_i^1}w_i^{(I')0}\cdot y^{I'},\ldots)
\]
into the quotient field $L$ of $A$. In each defining relation for $A$, we can solve for $w_i^{(I_i^1)0}$ in terms of the $y_i$'s and the other $w_i^{(I')0}$'s. After eliminating each $w_i^{(I_i^1)0}$ in this way, we see that $A$ is a polynomial algebra over $k(\{v^{(I)}_i\}_0, y_1,\ldots, y_{n+r})$. The $y_i$  and the $w_i^{(I')0}$, after removing $w_i^{(I_i^1)0}$ for each $i$, therefore form an algebraically independent set of generators for $L$ over $k(\{v^{(I)}_i\}_0)$,  and thus $K$ is a  pure transcendental extension of $k(\eta')$.   As Chow groups are invariant under base-change by  purely transcendental field extensions,  this proves (2).
\end{proof}

\begin{lemma}\label{lem:Decomp}  Take  $\gamma$ in $\CH_n(\sX^{d_*, n}_{H,\eta}\times_{k(\eta)}\sX^{d_*, n}_\eta)$. Then there is a zero cycle $y$ on $\sX^{d_*, n}_{H,\eta}$ a proper closed subset $D'$ of $\sX^{d_*, n}_\eta$ and a cycle $\gamma'$   supported on $\sX^{d_*, n}_{H,\eta}\times_{k(\eta)}D'$ such that
\[
\gamma=y\times \sX^{d_*, n}_\eta+\gamma'
\]
in $\CH_n(\sX^{d_*, n}_{H,\eta}\times_{k(\eta)}\sX^{d_*, n}_\eta)$. Furthermore the degree of $y$ is divisible by $\prod_{i=1}^rd_i$. 
\end{lemma}

\begin{proof} Let $\xi$ denote the generic point of $\sX^{d_*, n}_\eta$. By Lemma~\ref{lem:ChowGp}(3), the class of the restriction $j^*\gamma$ of $\gamma$ to $\sX^{d_*, n}_{H,\eta}\times_{k(\eta)}\xi$ is of the form
\[
j^*\gamma=M\cdot L\cdot \sX^{d_*, n}_{H,\eta}\times_{k(\eta)}\xi,
\]
where $L$ is a linear subspace of $H\subset \P^{n+r}$, $M$ an integer. Letting $y\in \CH_0(\sX^{d_*, n}_{H,\eta})$ be the 0-cycle $M\cdot L\cdot \sX^{d_*, n}_{H,\eta}$, the result follows from the localization theorem for the Chow groups; the assertion on the degree follows from the fact that $\sX^{d_*, n}_{H,\eta}$ has degree $\prod_{i=1}^rd_i$ and hence $y$ has degree $M\cdot\prod_{i=1}^rd_i$. 
\end{proof}

\begin{definition} For a natural number $n\ge1$, we let $n!^*$ denote the least common multiple of the numbers $1, 2,\ldots, n$. 
\end{definition}

\begin{lemma}\label{lem:Product} Let $d_1,\ldots, d_r$ be a sequence of positive natural numbers. Then the product $\prod_{i=1}^r(d_i!^*)$ is equal to the least common multiple $M$ of all products $i_1\cdot\ldots\cdot i_r$ with $1\le i_j\le d_j$, $j=1,\ldots, r$. 
\end{lemma}

\begin{proof}  Fix a prime number $p$. For each $j=1,\ldots, r$, let $i^*_j$ be an integer with $1\le i^*_j\le d_j$ and with $p$-adic valuation $\nu_p(i^*_j)$ equal to  $\nu_p(d_j!^*)$. Then
\[
\nu_p(\prod_{j=1}^ri^*_j)=\nu_p(\prod_{i=1}^r(d_i!^*))
\]
and $\nu_p(\prod_{j=1}^ri_j)\le \nu_p(\prod_{j=1}^ri^*_j)$ for all sequences $i_1,\ldots, i_r$ with $1\le i_j\le d_j$.  Thus $\nu_p(M)= \nu_p(\prod_{i=1}^ri^*_j)=\nu_p(\prod_{i=1}^r(d_i!^*))$. Since $p$ was arbitrary, this gives 
$M=\prod_{i=1}^r(d_i!^*)$.  \end{proof}

 \begin{theorem} \label{theorem:main1} For integers $d_1, \ldots, d_r$ with $\sum_id_i\le n+r$,  $\prod_{i=1}^rd_i!^*$ divides $\Tor_{k(\eta)}(\sX^{d_*, n}_\eta)$.
 \end{theorem}
 
 \begin{proof} We may suppose that $d_1>1$. Let $d_*'=(d_1-1, d_2,\ldots, d_r)$.    Let $\sO$ be the local ring of the origin in $\A^1_{k(\eta)}=\Spec k(\eta)[t]$ and let $\tilde\sX$ be the subscheme of $\P^{N_1}\times\ldots\times\P^{N_r}_\sO$ defined by the homogeneous ideal $(f_1,\ldots,f_r)$, with
 \[
 f_j=\begin{cases} \sum_{I\in \sS_{d_j, n+r}}u_j^{(I)}x^I&\text{ for } j\neq 1\\
  t\cdot \sum_{I\in \sS_{d_1, n+r}}u_1^{(I)}x^I+(1-t)\cdot x_{n+r}\cdot \sum_{J\in \sS_{d_1-1, n+r}}u_1^{(J)}x^J
 & \text{ for } j= 1.
  \end{cases}
 \]
The generic fiber of $\sX$ is thus isomorphic to $\sX^{d_*, n}_\eta\times_{k(\eta)}k(\eta, t)$ and the special fiber is $\sX^{d'_*, n}_\eta\cup H$. 

Suppose that $\sX^{d_*, n}_\eta$ admits a decomposition of the diagonal of order $N$:
\[
N\cdot \Delta_{\sX^{d_*, n}_\eta}=x\times \sX^{d_*, n}_\eta+\gamma
\]
with $\gamma$ supported on $\sX^{d_*, n}_\eta\times D$ for some divisor $D$. By Lemma~\ref{lem:ChowGp}, $\text{deg }x$ is divisible by $\prod_{i=1}^rd_i$,  and thus $\prod_{i=1}^rd_i$ divides $N$. 

By applying Totaro's specialization lemma (Lemma~\ref{lem:Spec2}) to the family $\sX\to\Spec\sO$, the diagonal for $\sX^{d'_*, n}_\eta$ admits a decomposition of the form
\[
N\cdot \Delta_{\sX^{d'_*, n}_\eta}=\bar{x}\times\sX^{d'_*, n}_\eta+ \gamma_1+\gamma_2
\]
with $\gamma_1$ supported in $\sX^{d'_*, n}_{H, \eta}\times \sX^{d'_*, n}_\eta$ and $\gamma_2$ supported in $\sX^{d'_*, n}_\eta\times D_2$ for some divisor $D_2$ on $\sX^{d'_*, n}_\eta$. By Lemma~\ref{lem:Decomp}, we have the identity
\[
\gamma_1=y\times \sX^{d'_*, n}_\eta+\gamma_3
\]
with $y$ a zero-cycle on $\sX^{d'_*, n}_\eta$ and $\gamma_3$ supported on $\sX^{d'_*, n}_\eta\times D_3$ for some divisor $D_3$. Thus, the diagonal on $\sX^{d'_*, n}_\eta$ admits a decomposition of order $N$ as well. By induction $(d_1-1)!^*\cdot\prod_{i=2}^r(d_i!^*)$ divides $N$; by symmetry $(d_j-1)!^*\cdot\prod_{i=1, i\neq j}^r(d_i!^*)$ divides $N$ for all $j$ with $d_j>1$. As we have already seen that $\prod_id_i$ divides $N$, Lemma~\ref{lem:Product} completes the proof.
\end{proof}

We also have a lower bound for the generic complete intersection with a rational point.

\begin{corollary} \label{corollary-with-point} For integers $d_1,\ldots, d_r$ with $\sum_id_i\le n+r$, let $K$ be the function field of the generic complete intersection of multi-degree $d_1,\ldots, d_r$, $K:=k(\eta)(\sX^{d_*, n}_\eta)$. Then $(1/\prod_{i=1}^rd_i)\prod_{i=1}^r(d_i!^*)$ divides $\Tor_K(\sX^{d_*, n}_\eta\times_{k(\eta)}K)$.
\end{corollary}

\begin{proof} Let $X=\sX^{d_*, n}_\eta$. By Lemma~\ref{lem:ChowGp}, $I_X=\prod_{i=1}^rd_i$ and thus by Lemma~\ref{lem:Index}, $\Tor_{k(\eta)}(\sX^{d_*, n}_\eta)$ divides $I_X\cdot\Tor_K(\sX^{d_*, n}_\eta\times_{k(\eta)}K)$.
Clearly $\prod_{i=1}^rd_i$ divides $\prod_{i=1}^r(d_i!^*)$, whence the result.
\end{proof}

\begin{ex}[Generic cubic hypersurfaces] For the generic cubic hypersurface $X:=\sX^{3, n}_\eta$, $n\ge2$, we thus have $\Tor_{k(\eta)}X=6$ and the generic cubic hypersurface with a rational point $X_K$, $K=k(\eta)(X)$, has $2|\Tor_{K}X_K|6$. If $X_K$ were to admit a dominant rational map  $\P^n\dashrightarrow X_K$ of degree prime to 3, then by Proposition~\ref{prop:Ratl}(2), we would have $\Tor_{K}X_K=2$.   We know that if a cubic hypersurface $X$ has a line (defined over the base-field) then there is a degree two  dominant rational map  $\P^n\dashrightarrow X$ (see for example \cite[\S 5]{Murre}), but it is not clear if this is the case if we only assume that $X$ has a (suitably general) rational point.

However, as pointed out by a referee, the generic cubic {\em surface} with a rational point does have $\Tor_{K}X_K=6$, at least if $k$ has characteristic not equal to 3. Indeed, if we take a field $k_0$ of characteristic different from three, containing a primitive cube root of 1, and let $k$ be a pure transcendental extension of $k_0$, we may find an element $a\in k$ that is not a cube. Then the smooth cubic surface $Y\subset \P^3_k$ given by $x^3+y^3+z^3+at^3=0$ has a rational point but also has $Br(Y)/Br(k)\cong (\Z/3)^2$ (see for example \cite{Manin}), and thus $\Tor_k(Y)$ is divisible by 3. Specializing the generic cubic surface  with a rational point $X_K$ to $Y$, we may apply the divisibility lemma~\ref{lem:specialization} to conclude that $3|\Tor_{K}X_K$. In particular, the generic cubic surface with a rational point does not admit a rational map $\P^2\dashrightarrow X_K$ of degree not divisible by 6.
\end{ex}

\begin{ex}[Generic cubic hypersurfaces with a line]\label{ex:line} Take $n\ge2$. For $X$ a cubic hypersurface in $\P^{n+1}_L$ (defined over some field $L\supset k$), we have the Fano variety of lines on $X$, $F_X$, a closed subscheme of the Grassmann variety $\Gr(2, n+2)_L$. In fact, if $U\to \Gr(2, n+2)$ is the universal rank two bundle, and $f$ is the defining equation for $X$,   then $F_X$ is the closed subscheme defined by the vanishing of the section of the rank four bundle $\Sym^3 U$ determined by $f$. In particular, the class of $F_X$ in $\CH^4(\Gr(2, n+2)_L)$ is given by the Chern class $c_4(\Sym^3 U)$. One computes this easily as $c_4=9c_2^2(U)+18c_1(U)^2c_2(U)$.  As $c_2(U)^n$ and $c_2(U)^{n-2}c_1(U)^2$ both have degree one, we see that $F_X\cdot c_2(U)^{n-2}$ has degree 27, and thus $I_{F_X}$ divides 27. This 27 is of course the famous 27 lines on a cubic surface, as intersecting $F_X$ with $c_2(U)^{n-2}$ in $\Gr(2,n+2)$ is the same as taking the Fano variety of the intersection of $X$ with a general $\P^3$ in $\P^{n+1}$. See for example \cite[14.7.13]{Fulton} for details of the Chern class computation.

Taking $X=\sX^{3, n}_\eta$, and letting $K=k(\eta)(F_X)$, it follows from Lemma~\ref{lem:Index} that  $6=\Tor_{k(\eta)}(\sX^{d_*, n}_\eta)$ divides $27\cdot\Tor_K(\sX^{d_*, n}_\eta\times_{k(\eta)}K)$; since we have the degree two rational map $\P^n_K\dashrightarrow \sX^{d_*, n}_\eta\times_{k(\eta)}K$, we have $\Tor_K(\sX^{d_*, n}_\eta\times_{k(\eta)}K)=2$. In particular, the generic cubic with a line is not stably rational over its natural field of definition $k(\eta)(F_X)$. 
\end{ex}

We are indebted to J.-L. Colliot-Th\'el\`ene for the next example (see \cite[Th\'eor\`eme]{CTNew}), which improves the bounds and simplifies the argument of an example in an earlier version of this paper.
\begin{ex}[Cubics over a ``small'' field] \label{ex:small} Take $n\ge2$. We consider a DVR $\sO$ with quotient field $K$ and residue field $k$ (of characteristic $\neq2$), and a degree $3$ hypersurface $\sX\subset \P^{n+1}_\sO$. Let $X=\sX_K$ and $Y=\sX_k$. We suppose that $X$ is  smooth and $Y=Q\cup H$, with $Q$ a smooth quadric and $H$ a hyperplane. Furthermore, we assume
\begin{enumerate}
\item $I_Q=1$.
\item  $Q$ and $H$ intersect transversely.
\item $I_{Q\cap H}=2$.
\end{enumerate}
From Proposition~\ref{prop:UpperBound}, we know that $\Tor_K(X)$ is finite and divides 6. We will show that 2 divides $\Tor_K(X)$.

For this, suppose we have a decomposition of the diagonal of $X$ of order $N$. We note that our family $\sX$ satisfies the hypotheses of Lemma~\ref{lem:Spec3}, with $Y_1=Q$, $Y_2=H$, and $r=1$. By Remark~\ref{rem:Spec3}, $N\cdot(\CH_0(Q)/i_{Q\cap H*}(\CH_0(Q\cap H))=0$; considering degrees, we see that $2|N$. 

To construct an explicit example, recall \cite{Lam} that a field $k$ has {\em $u$-invariant}  $u(k)\ge r$ if there exists an anisotropic quadratic form over $k$ of dimension $r$. The above construction gives us a cubic hypersurface $X$ of dimension $n\ge 2$ over $K:=k((x))$ with $2 |\Tor_K(X)$ and $X(K)\neq\0$ if $k$ is an infinite field of characteristic $\neq2$ with $u$-invariant $\ge n+1$. Indeed, take a anisotropic quadratic form $q_0$ in $n+1$-variables $X_0,\ldots, X_n$,  choose $\alpha\in k^\times$ represented by $q_0$ and let $q=q_0-\alpha\cdot X^2_{n+1}$, so $q$ is non-degenerate. Let $Q\subset\P^{n+1}_k$ be the quadric defined by $q$ and let $H$ be the hyperplane $X_{n+1}=0$. Take a cubic form $c_0\in k[X_0,\ldots, X_{n+1}]$ and let $c=xc_0+q\cdot X_{n+1}\in k[[x]][X_0,\ldots, X_{n+1}]$. Since $k$ is infinite, we can choose $c_0$ so that the subscheme $X$ of $\P^{n+1}_{k((x))}$ defined by $c$ is smooth (and hence geometrically integral); it suffices to choose $c_0$ so that $c_0=0$ is smooth and intersects $Q$ and $H$ transversely. Clearly $I_Q=1$, $Q$ and $H$ intersect transversely and $I_{Q\cap H}=2$, giving us the desired example.  

For example, $\F_p$ has $u$-invariant 2, and $\Q_p$ and $\F_p((t))$  both have $u$-invariant 4 (see for example  \cite{Lam}).  Thus there are cubic threefolds $X$ over $K:=\Q_p((x))$ with $2|\Tor_K(X)$ and with $X(K)\neq\0$. Similarly, there are examples of such cubic threefolds over $K=\F_p((t))((x))$ for $p\neq 2$.  Over $K=\Q((x))$ or even over $K=\R((x))$ there are cubic hypersurfaces $X$ of   dimension $n$  over $K$ for arbitrary $n\ge2$, with $2 |\Tor_K(X)$ and $X(K)\neq\0$. As in the previous example, we may pass to an odd degree field extension $L$ of $K$ to find a cubic hypersurface $X_L$ with a line, and with $\Tor_L(X_L)=2$; all these cubics are thus not stably rational over their corresponding field of definition. 
\end{ex}

\begin{remark} As mentioned in the introduction, Colliot-Th\'el\`ene and Pirutka have constructed cubic threefolds over a $p$-adic field  \cite[Th\'eor\`eme 1.21]{CTP} and over $\F_p((x))$ \cite[Remarque 1.23]{CTP} with non-zero torsion order and having a rational point.\end{remark}

\section{Torsion order for complete intersections in a projective space: a lower bound}\label{sec:LowerBound1} As in the previous sections, we consider smooth complete intersection subschemes $X$ of $\P^{n+r}$ of multi-degree $d_1,\ldots, d_r$. 

By saying a property holds for  a very general complete intersection in $\P_k^{n+r}$ of multi-degree $d_1,\ldots, d_r$ we mean that there is a countable union $F$ of proper closed subsets of the parameter scheme of  such complete intersections 
(an open in a product of projective spaces over $k$)  such that the property holds for $X_b$ if $b\not\in F$.

Recall that for $X$ a proper, separable $L$-scheme for some field $L$, and $\bar L$  the algebraic closure of $L$, we have defined $\Tor^{(i)}(X):=\Tor^{(i)}_{\bar L}(X_{\bar L})$.

\begin{theorem}\label{thm:LowerBound} Let $k$ be a field of characteristic zero.  Let $d_1,\ldots, d_r$ and $n\ge3$ be integers with $d':=\sum_{j=1}^rd_j\le n+r$. Let $p$ be a prime number. Suppose that
\begin{equation}\label{eqn:*}
d_i \geq p\cdot \left\lceil {\frac{n+r+1-d'+d_i}{p+1}} \right\rceil 
\end{equation}
for some $i$, $1\le i\le r$. Then $p | \Tor^{(n-2)}(X)$ for all very general $X=X_{d_1,\ldots, d_r}\subset \P_k^{n+r}$.
\end{theorem} 

\begin{corollary}\label{cor:LowerBound} Let $k$, $d_1,\ldots, d_r$,   $n$ and $p$ be as in Theorem~\ref{thm:LowerBound} and suppose that $d_i$ satisfies \eqref{eqn:*}. Then $p | \Tor(X)$ for all very general $X=X_{d_1,\ldots, d_r}\subset \P_k^{n+r}$.
\end{corollary}

\begin{proof}   $\Tor^{(n-2)}(X)$  divides $\Tor(X):=\Tor^{(0)}(X)$ by Lemma~\ref{lem:Basic}(1).
\end{proof}

\begin{remarks}
1.  We know that $\Tor(X)$ is finite for all  $X=X_{d_1,\ldots, d_r}\subset \P^{n+r}$ with $\sum_jd_j\le n+r$ by Proposition~\ref{prop:UpperBound} and hence $\Tor^{(n-2)}(X)$ is also finite.
\\
2. For $p=2$ and for hypersurfaces, the corollary follows directly from the results in Totaro's paper \cite{Totaro}. \\
3. We only use the hypothesis of characteristic zero to allow for a specialization to characteristic $p$, where $p$ is the prime number in the statement. For $k$ a field of positive characteristic, the analogous result holds, but only for $p=\Char k$.\\
4. There are two interesting cases of complete intersection threefolds we would like to mention: that of a multi-degree $(3,2)$ complete intersection in $\P^5$ and a multi-degree $(2,2,2)$ complete intersection in  $\P^6$ (see the recent   results of Hassett-Tschinkel \cite{HT}). In both cases we take $d_i=2$ and get a divisibility by 2.
Notice that in the $(2,3)$ case taking $d_i=3$ and $p=3$ works.
\end{remarks}

\begin{proof}[Proof of the theorem] This is another application of the argument of Koll\'ar \cite{Kollar}, as used for example by Totaro \cite{Totaro}, Colliot-Th\'el\`ene and Pirutka \cite{CTP2}, or Okada \cite{O}. We may reorder the $d_j$ so that $d_i=d_1$. We first assume that $p$ divides $d_1$,   $d_1=q\cdot p$. Take $f$ and $g$ suitably general homogeneous polynomials of degree $d_1$ and $q$, respectively, and let $f_2,\ldots, f_r$ be suitably general homogeneous polynomials, with $f_j$ of degree $d_j$, $j=2,\ldots, r$. We take these to be in the polynomial ring $\sO[X_0,\ldots, X_{n+r}]$, where $\sO$ is a complete (hence excellent)  discrete valuation ring with maximal ideal $(t)$,  residue field $k=\bar{\F}_p$, the algebraic closure of $\F_p$, and with quotient field $K$ a field of characteristic zero. We let $\sX\to \Spec \sO$ be the closed subscheme of a weighted projective space $\P=\Proj \sO[X_0,\ldots, X_{n+r}, Y]$, with the $X_i$ having weight 1 and $Y$ having weight $q$, defined by the homogeneous ideal
\[
(f_2,\ldots, f_r, Y^p-f, g-tY).
\]
The generic fiber $X:=\sX_K$ is isomorphic to the complete intersection subscheme of $\P^{n+r}_K$ defined by $g^p-t^pf=f_2=\ldots=f_r=0$ and the special fiber $Y:=\sX_k$ is the cyclic $p$ to 1 cover $Y\to W$, with $W\subset \P^{n+r}_k$ the complete intersection defined by $\bar{g}=\bar{f}_2=\ldots=\bar{f}_r=0$, and $y^p=f_{|W}$.

For general $f, g, f_2,\ldots, f_r$, $X$ and $W$ are smooth, and $Y$ has only finitely many singularities, which may be resolved by an explicit iterated blow-up $q:Z\to Y$ which is totally $\CH_0$-trivial: for details, see Proposition~\ref{proposition:good-sections} if $p\geq 3$. If $p=d_1=2$,  then we use Lemma~\ref{lemma-2-non-degenerate} and Proposition~\ref{proposition:resolution} for the construction of the resolution of singularities and the proof that the resolution morphism $q$ is totally $\CH_0$-trivial. Koll\'ar shows in addition, that under the assumption \eqref{eqn:*}, one has $H^0(Z, \Omega^{n-1}_{Z/k})\neq\{0\}$. In somewhat more detail, Koll\'ar  (see \cite[\S 15, Lemma 16]{Kollar} defines an invertible sheaf $Q$ (denoted $\pi^*Q(L,s)$ in {\it loc.~cit.}) with an injection $Q\to (\Omega^{n-1}_{Y/k})^{**}$, where ${}^{**}$ denotes the double dual. A local computation (see \cite{CTP2}, \cite{O} or Remark \ref{remark-m=1} for details) in a neighborhood of the finitely many singularities of $Y$ shows that this injection extends to an injection $q^*Q\to \Omega^{n-1}_{Z/k}$; here is where the condition $n\ge3$ is used. In addition, $q^*Q$  is isomorphic to the pullback to $Z$ of $\omega_W\otimes\sO_W(d_1)$, where $\omega_W$ is the canonical sheaf on $W$. As $\omega_W=\sO_W(d_1/p+\sum_{j\ge2}d_j-n-r-1)$,  we have a non-zero section of  $\Omega^{n-1}_{Z/k}$  if $d_1(p+1)/p\ge n+r+1-\sum_{i=2}^rd_i$, which is exactly the condition in the statement of the theorem.

By   Proposition~\ref{prop:UpperBound}, we know that $\Tor(X_{\bar K})$ is finite and thus 
 $\Tor^{(n-2)}(X_{\bar K})$ is finite as well. The specialization result Proposition~\ref{prop:specialization2} thus implies that $\Tor^{(n-2)}(Z_{\bar k})$ is finite and divides $\Tor^{(n-2)}(X_{\bar K})$. By \cite[Prop.~4.2.33]{Gros}, \cite[Thm.~3.1.8]{CR1}, and \cite[III.3.Prop.~4]{EZ}, correspondences on $Z\times_k Z$ act on $H^0(Z, \Omega^{n-1}_{Z/k})$ and if $\gamma$ is a correspondence on $Z\times_k Z$ supported in some $Z'\times Z$ with $\dim_k Z'\le n-2$, then by \cite[Proposition~3.2.2(2)]{CR1}, $\gamma_*$ acts by zero on $H^0(Z, \Omega^{n-1}_{Z/k})$. Similarly, if $\gamma$ is a correspondence on $Z\times Z$, supported in $Z\times D$ for some divisor $D\subset Z$, then  $\gamma_*(\omega)_{|Z\setminus D}=0$ for each $\omega \in H^0(Z, \Omega^{n-1}_{Z/k})$; as  $\Omega^{n-1}_{Z/k}$ is locally free, it follows that $\gamma_*(\omega)=0$. Thus, if  $\Delta_Z$ admits a decomposition of order $N$ and level $n-2$, this implies that $N\cdot \omega=0$ for all $\omega\in H^0(Z, \Omega^{n-1}_{Z/k})$, and since $H^0(Z, \Omega^{n-1}_{Z/k})$ is a non-zero $k$-vector space, this implies that $p|N$. Since $\Tor^{(n-2)}(Z_{\bar k})$  divides $\Tor^{(n-2)}(X_{\bar K})$, it follows that $p|\Tor^{(n-2)}(X_{\bar{K}})$ and Corollary~\ref{cor:specialization} finishes the proof in this case.

 In the case of a general $d_1$, write $d_1=q\cdot p+c$, $0<c<p$, and consider a family $\sX\to \Spec \sO$ defined by a homogeneous ideal of the form
\[
(f_2,\ldots, f_r, (Y^p-h)s+tu, g-tY),
\]
with $u, h, g, s\in \sO[X_0,\ldots, X_{n+r}]$, $u$ of degree $d_1$, $h$ of degree $pq$, $g$ of degree $q$ and $s$ of degree $c$, suitably general, and with $Y$ as above of weight $q$. The generic fiber $X$ is the complete intersection $f_1=f_2=\ldots=f_r=0$, with $f_1=(g^p-t^ph)s+t^{p+1}u$; the special fiber $Y$ has two components $Y_1, Y_2$, with $Y_1$ the $p$ to 1 cyclic cover of 
$W:=(\bar{f}_2=\ldots=\bar{f}_r=\bar{g}=0)$, branched along $W\cap(h=0)$. We take $q:Z\to Y_1$ the resolution as in the previous case. Having chosen $h, g, s$, we may take $u$ sufficiently general so that $X$ is a smooth complete intersection. 

Since $\sO$ is excellent, we are free to make a finite extension $L$ of $K$, take the integral closure $\sO_L$ of $\sO$ in $L$, replace $\sO$ with the localization  $\sO'$ at a maximal ideal of $\sO_L$, and replace $\sX$ with $\sX\otimes_\sO\sO'$; changing notation, we may assume that $\Tor^{(n-2)}_K(X)$ is the geometric torsion order $\Tor^{(n-2)}(X)$. By Proposition~\ref{prop:specialization3},  the smooth proper $k$-scheme $Z$ admits a decomposition of the diagonal as
\[
N\cdot \Delta_Z=\gamma+\gamma_1+\gamma_2,
\]
with $N=\Tor^{(n-2)}(X)$, $\gamma$ supported in $Z_{n-2}\times Z$ with $\dim Z_{n-2}\le n-2$, $\gamma_1$ supported in $q^{-1}(Y_1\cap Y_2)\times Z$ and $\gamma_2$ supported in $Z\times D$ for some divisor $D$ on $Z$.

We may take the degree $c$ part $s$ as general as we like. In particular, we may assume that $Y_1\cap Y_2$ is contained in the smooth locus of $Y_1$ and is thus isomorphic to a closed subscheme $Z'$ of $Z$. 
 
Our  decomposition of the diagonal on $Z$ gives the relation
\[
N\cdot \omega=\gamma_{1*}\omega
\]
for each $\omega\in H^0(Z,\Omega^{n-1}_Z)$. Indeed, 
\[
N\cdot\omega=N\cdot\Delta_{Z*}\omega=\gamma_{1*}\omega+\gamma_{2*}\omega+\gamma_{*}\omega.
\]
But $\gamma_{*}$ factors through the restriction to $Z_{n-2}$, so $\gamma_*\omega=0$. Similarly, $\gamma_{2*}\omega$ is a global section of $\Omega^{n-1}_Z$ supported in $D$, which is zero, since $\Omega^{n-1}_Z$ is a locally free sheaf.

One computes  that the canonical class of $Y_1\cap Y_2$ is  anti-ample and thus the canonical line bundle on the dimension $n-1$ subscheme $Z'$ has no sections. Note that $Z'$ is a cyclic $p$ to $1$ cover of the complete intersection $W\cap V(\bar{s})$. If $s$ is general then there is a rational resolution of singularities $\tilde{Z}'$ (Proposition \ref{proposition:resolution}, Lemma \ref{lemma:rational-resolution}), hence the canonical line bundle of $\tilde{Z}'$ has no non-vanishing sections. But $\gamma_{1*}\omega$ factors through the restriction of $\omega$ to $\tilde{Z}'$, hence $\gamma_{1*}\omega=0$. Since $h$ has degree $q\cdot p$ in the range needed to give the existence of a non-zero $\omega$ in $H^0(Z,\Omega^{n-1}_Z)$, we conclude as before that $p|N$. 
\end{proof}

\begin{ex} We consider the case of hypersurfaces of degree $d$ in $\P^{n+1}$, $n\ge3$. The theorem says that $p$ divides $\Tor^{(n-2)}(X)$ for very general degree $d\le n+1$ hypersurfaces $X$ in $\P^{n+1}$ if
 \[
 d \geq p\cdot \left\lceil {\frac{n+2}{p+1}} \right\rceil 
 \]
For $p=2$, this is the range considered by Totaro; for $p=3$, the first case is degree 6 in $\P^6$. For the extreme case of degree $d=n+1$ in $\P^{n+1}$, we have $p|\Tor^{(n-2)}(X)$ for all $p$ dividing $n+1$.  
 \end{ex}

\section{An improved lower bound for the very general complete intersection}\label{sec:LowerBound2}
 
In this section we extend Theorem~\ref{thm:LowerBound} to cover prime powers. The basic idea is to replace the differential forms with  Hodge-Witt cohomology. We are grateful to the referee for providing the argument for the next lemma which shows  that a cycle on $Z\times Z$,  supported on $Z'\times Z$ with $\dim Z'\le n-2$, acts trivially on $H^0(Z, W_m\Omega^{n-1}_Z)$. 

\begin{lemma}\label{lemma-advanced-technology-for-Wm}
Let $k$ be a perfect field of positive characteristic $p$, and $X,Y$ smooth, equidimensional, and quasi-projective $k$-schemes. Set $n=\dim X$ and ${\rm CH}^n_{{\rm prop}/Y}(X\times Y)=\varinjlim_Z {\rm CH}_{\dim Y}(Z)$, where the limit is over all closed subsets $Z\subset X\times Y$ that are proper over $Y$. For $\alpha \in {\rm CH}^n_{{\rm prop}/Y}(X\times Y)$ denote by 
$$
\alpha_*: \oplus_{i,j} H^i(X,W_m\Omega^j) \xr{} \oplus_{i,j} H^i(Y,W_m\Omega^j)
$$
the map induced by $\alpha$ via the cycle action from \cite{CR-drw}. Assume $\alpha$ is supported on $A\times Y$, where $A\subset X$ is a closed subset of codimension $\geq r$. Then $\alpha_*$ vanishes on $\oplus_{i,j+r>n} H^i(X,W_m\Omega^j)$.
\begin{proof}
We may assume $\alpha=[Z]$, with $Z\subset X\times Y$ an integral closed subscheme of codimension $n$ supported on $A\times Y$. Denote by $p_X,p_Y$ the respective projections from $X\times Y$. It suffices to show for $i\geq 0, j+r> n$, and $b\in H^i(X,W_m\Omega^j)$ that 
\begin{equation}\label{equation-Z-kills}
p_X^*(b)\cup {\rm cl}[Z] = 0 \quad \text{in $H^{i+n}_Z(X\times Y,W_m\Omega^{j+n}_{X\times Y})$.}
\end{equation}
Then $\alpha_*(b)=p_{Y*}(p_X^*(b)\cup {\rm cl}[Z])$ will also vanish.

We first prove \eqref{equation-Z-kills} for $i=0$. Denote by $\eta\in X\times Y$ the generic point of $Z$. Since $W_m\Omega_{X\times Y}^{j+n}$ is Cohen-Macaulay the natural map 
$
H^n_Z(X\times Y,W_m\Omega_{X\times Y}^{j+n}) \xr{} H^n_{\eta}(X\times Y,W_m\Omega_{X\times Y}^{j+n})
$
is injective. Set $B=\OO_{X\times Y,\eta}$ and $C=\OO_{X,p_X(\eta)}$; by assumption we have $\dim C\geq r$. Since $B$ is formally smooth over $C$ we find $t_1,\dots,t_r\in C$ and $s_{r+1},\dots,s_n\in B$ such that $p_X^*(t_1),\dots,p_X^*(t_r),s_{r+1},\dots,s_n$ form a regular sequence of parameters of $B$. Hence by \cite[II, 3.5]{Gros} (see also \cite[Prop.~2.4.1]{CR-drw}), \cite[Lem.~3.1.5]{CR-drw} and in the notation of \cite[1.11.1]{CR-drw} the image of $p_X^*(b)\cup {\rm cl}[Z]= \Delta^*(p_X^*(b)\times {\rm cl}[Z])$ in  $H^n_{\eta}(X\times Y,W_m\Omega_{X\times Y}^{j+n})$ is up to a sign given by 
$$
{p_X^*(b\cdot d[t_1]\cdots d[t_r]) \cdot d[s_{r+1}]\cdots d[s_n] \brack p_X^*([t_1]),\dots,p_X^*([t_r]),[s_{r+1}],\dots,[s_n]}.
$$
Hence the vanishing follows from $b\cdot d[t_1]\cdots d[t_r]\in W_m\Omega_X^{j+r}=0$. 

For the general case $i\geq 0$, we first observe that the CM property of $W_m\Omega_{X\times Y}^{j+n}$ implies $R\underline{\Gamma}_Z(W_m\Omega_{X\times Y}^{j+n})\cong \mathcal{H}^n_Z(W_m\Omega_{X\times Y}^{j+n})[-n]$. Therefore 
$$
H^{i+n}_Z(X\times Y, W_m\Omega_{X\times Y}^{j+n})= H^i(X\times Y,\mathcal{H}^n_Z(W_m\Omega_{X\times Y}^{j+n})).
$$
Let $\mc{U}$ be an open affine cover of $X$ and denote by $\mc{U}\times Y$ the open (not necessarily affine) cover of $X\times Y$. We can consider the Cech cohomology with respect to $\mc{U}\times Y$ and obtain a natural map
\begin{equation}\label{equation-cech-support}
\check{H}^i(\mc{U}\times Y, \mathcal{H}^n_Z(W_m\Omega_{X\times Y}^{j+n})) \xr{} H^i(X\times Y, \mathcal{H}^n_Z(W_m\Omega_{X\times Y}^{j+n})).
\end{equation}
Since $\check{H}^i(\mc{U},W_m\Omega^j_X)=H^i(X,W_m\Omega^j_X)$ and pullback and cup product are compatible with restriction to open subsets, we see that $p_X^*(-)\cup {\rm cl}[Z]:H^i(X,W_m\Omega^j_X) \xr{}  H^{i+n}_Z(X\times Y, W_m\Omega^{j+n}_{X\times Y})$ naturally factors via \eqref{equation-cech-support}. Therefore the case $i\geq 0$ follows from the case $i=0$.       
\end{proof}
\end{lemma}
 
\begin{theorem}\label{thm-main-alg-closed} Let $k$ be a field of characteristic zero. 
  Let $X\subset \P^{n+r}_k$ be a very general complete intersection of multi-degree  $d_1, d_2, \dots, d_r$ such that $d':=\sum_{i=1}^r d_i \leq n+r$ and $n\geq 3$. Let $p$ be a prime,  $m\geq 1$, and suppose
\begin{equation}\label{equation:inequality}
d_i \geq p^{m}\cdot \left\lceil {\frac{n+r+1-d'+d_i}{p^m+1}} \right\rceil 
\end{equation}
for some $i$. Furthermore, we suppose that $p$ is odd or $n$ is even. Then $p^m| \Tor^{(n-2)}(X)$.
\end{theorem}

\begin{remark} Just as for Theorem~\ref{thm:LowerBound}, the same result holds for $k$ a field of positive characteristic, but only for $p=\Char k$.
\end{remark}

\begin{proof}   
The proof relies on Theorem~\ref{thm:cyclic-covers}, which we prove later in this section.

By Corollary~\ref{cor:specialization}, we need to find only one smooth complete intersection $X\subset  \P^{n+r}_k$ such that $p^m| \Tor^{(n-2)}(X)$.

For a scheme $X$ with locally free sheaf $\sE$ and a section $s:\sO_X\to \sE$, we let $V(s)$ denote the closed subscheme of $X$ defined by $s$.  

We set $d=d_i$, $a=\left\lceil {\frac{n+r+1-d'+d}{p^m+1}} \right\rceil $, and $c=d-p^{m}\cdot a$. Let $\sO=W(\bar{\mathbb{F}}_p)$ and $K={\rm Frac}(\sO)$, we take $r$, $f$, $g$, $l$, and $f_2,\dots,f_r$ suitably general (we will make this precise) homogeneous polynomials in $\sO[X_0,\ldots, X_{n+r}]$ of degree $d$, $d-c$, $a$, $1$, and $d_2,\dots,d_r$, respectively. We let $\sX\to \Spec \sO$ be the closed subscheme of the weighted projective space $\P=\Proj\; \sO[X_0,\ldots, X_{n+r}, Y]$, with the $X_i$ having weight $1$, and $Y$ having weight $a$, defined by the homogeneous ideal
\begin{equation}
l^c\cdot (Y^{p^m}-f) + p\cdot r, g-p\cdot Y, f_2,\dots, f_r.
\end{equation}
The generic fiber $X:=\sX_K$ is isomorphic to the complete intersection of $\P^{n+r}_K$ defined by $l^c\cdot (g^{p^m}-p^{p^{m}}\cdot f) + p^{p^m+1}\cdot r,f_2\dots,f_r$. For $r,f_2,\dots,f_r$ general, it is smooth.  
By replacing $\sO$ with its normalization in a suitable finite extension of $K$ and changing notation, we may assume that $\Tor^{(n-2)}_K(X)$ is equal to the geometric torsion order $\Tor^{(n-2)}(X)$.

The special fiber $Y:=\sX_{\bar{\mathbb{F}}_p}$ is $Y=Y_1+ c\cdot Y_2$. Here, $Y_1$ is the cyclic $p^m$ cover $Y_1\to W$ defined by $f\in H^0(W,\sO(a)^{\otimes p^m})$, with $W\subset \P^{n+r}_{\bar{\mathbb{F}}_p}$ the complete intersection defined by $g,f_2,\dots,f_r$. We will take $f,g,f_2,\dots,f_r$ general enough so that 
\begin{enumerate}
\item $W$ is smooth,
\item $Y_1$ has non-degenerate singularities (see \S\ref{section:cyclic}), 
\item the assumption (3) of Theorem~\ref{thm:cyclic-covers} is satisfied for $Y_1$. 
\end{enumerate}
For (2) we use Proposition~\ref{proposition:good-sections} if $d-c\geq 3$. If $d-c=2$ hence $p=2$ then we use Lemma~\ref{lemma-2-non-degenerate}. For (3) we use the theorem of Illusie about ordinarity of a general complete intersections \cite{Illusie-ordinary}.
 Let us check that all other assumptions of Theorem~\ref{thm:cyclic-covers} are satisfied. (1) is evident, and (2) is equivalent to $(p^m+1)\cdot a - n - r - 1 +d' -d\geq 0$ which follows immediately from the definition of $a$. Assumption (4) is equivalent to $i \cdot a + a - n - r - 1 + d' - d < 0$, for all $i=0,\dots,p^m-1$, which follows from $d'<n+r+1$; (5) is obvious.

The variety $Y_2$ is defined by $l,g,f_2,\dots,f_r$, and only exists if $c\neq 0$. We take $l$ general so that $Y_2$ does not contain the singular points of $Y_1$, $W\cap V(l)$ is smooth, and the $p^m$ cyclic covering of $W\cap V(l)$ corresponding to $f_{| W\cap V(l)}$ has non-degenerate singularities.

Let $r:\tilde{Y}_1\xrightarrow{} Y_1$ be the resolution of singularities constructed in Proposition~\ref{proposition:resolution}; the map $r:\tilde{Y}_1\xr{} Y_1$ is totally $\CH_0$-trivial. By Proposition~\ref{prop:specialization3},
$$
\Tor^{(n-2)}(X)\cdot \Delta_{\tilde{Y}_1} = \gamma + Z + Z_2, 
$$
where $\gamma$ is a cycle with support in $A\times \tilde{Y}_1$ with $\dim A \leq n-2$, $Z$ has support in $\tilde{Y}_1\times D$ with $D$ a divisor, and $Z_2$ has support in $(Y_1\cap Y_2)\times \tilde{Y}_1$. 

In view of Theorem~\ref{thm:cyclic-covers}, we have $\Z/p^m\subset H^0(\tilde{Y}_1,W_m\Omega^{n-1})$. By the work \cite{CR-drw} on Hodge-Witt cohomology, we have an action of algebraic correspondences on $H^0(\tilde{Y}_1,W_m\Omega^{n-1})$ (relying on Gros' cycle class \cite{Gros}). 
Let us show that $Z_2$ acts trivially. Note that $T:=Y_1\cap Y_2$ is the $p^m$ cyclic covering of $W\cap V(l)$ corresponding to $f_{| W\cap V(l)}$. An easy computation shows $H^{>0}(Y_1\cap Y_2,\sO)=0$, hence $H^{>0}(\tilde{T},\sO)=0$ by Lemma~\ref{lemma:rational-resolution}, where  $\tilde{T}$ is the resolution constructed in Proposition~\ref{proposition:resolution}, and $H^{>0}(\tilde{T},W_m(\sO))=0$. By Ekedahl duality \cite{E}, we get $H^{<n-1}(\tilde{T},W_m\Omega^{n-1})=0$. Let $\tilde{Z}_2$ be a lift of $Z_2$ to $\tilde{T}\times \tilde{Y}_1$. The action of $Z_2$ factors
as
$$
H^0(\tilde{Y}_1,W_m\Omega^{n-1}) \xr{} H^0(\tilde{T},W_m\Omega^{n-1})  \xr{\tilde{Z}_2}
H^0(\tilde{Y}_1,W_m\Omega^{n-1}),  
$$
the first map being the pullback for the map $\tilde{T}\xr{} \tilde{Y}_1$, thus it is zero. 

Lemma \ref{lemma-advanced-technology-for-Wm} implies that the action of $\gamma$ on $H^0(\tilde{Y}_1,W_m\Omega^{n-1})$ vanishes. Therefore 
\begin{multline*}
H^0(\tilde{Y}_1,W_m\Omega^{n-1}) \xr{\Tor^{(n-2)}(X)\cdot} H^0(\tilde{Y}_1,W_m\Omega^{n-1}) \\ \xr{\text{restriction}} H^0(\tilde{Y}_1\backslash D,W_m\Omega^{n-1})
\end{multline*}
is zero. Since the restriction map is injective, we get $p^m|\Tor^{(n-2)}(X)$.
\end{proof}

\begin{corollary} Let $k$ be a field of characteristic zero. Let $X\subset \P^{n+r}_k$ be a very general complete intersection of multi-degree $(d_1,\dots,d_r)$ with $\sum_i d_i=n+r$ and $n\geq 3$. If $n$ is even or $d_i$ is odd then $d_i|\Tor^{(n-2)}(X)$.
\end{corollary}

\subsection{} \label{section:cyclic} 
Let $X$ be a smooth variety over an algebraically closed field $k$ of characteristic $p$. Suppose that $n:=\dim X\geq 2$.
Let $L$ be a line bundle on $X$, and let $s\in H^0(X,L^{\otimes p^m})$. We denote by 
$\pi:Y\xrightarrow{} X$ the $p^m$ cyclic covering corresponding to $s$. It is an inseparable morphism and induces an homeomorphism on 
the underlying  topological spaces.

There is a tautological connection $d:L^{\otimes p^m} \xrightarrow{} L^{\otimes p^m}\otimes \Omega_X^1$ which satisfies $d(t^{p^m})=0$ for all sections $t\in L$. 
In particular, we have $d(s)\in H^0(X,L^{\otimes p^m}\otimes \Omega_X^1)$. Note that $Y_{{\rm sing}}=\pi^{-1}(V(d(s)))$. 

We say that $Y$ has \emph{non-degenerate} singularities if the following conditions hold:
\begin{enumerate}
\item $Y$ has at most isolated singularities, or equivalently, $\dim(V(d(s)))=0$ or $V(d(s))=\emptyset$.
\item For all $x\in V(d(s))$, ${\rm length}(\sO_{V(d(s)),x})\leq 1$, if $p$ is odd or $p=2$ and $n$ is even. If $p=2$ and $n$ is odd then we require ${\rm length}(\sO_{V(d(s)),x})\leq 2$ and the blow up ${\rm Bl}_xY$ of $x$ has an exceptional divisor that is a cone over a smooth quadric.
\end{enumerate}

Around a non-degenerate singularity of $Y$, we can find local coordinates $x_1,\dots,x_n$ of $X$ such that $Y$ is defined by 
\begin{align}
&y^{p^m}+x_1^2+\dots+x_n^2+f_3 \qquad \text{if $p$ is odd}, \label{align-1}\\
&y^{p^m}+x_1x_2+\dots+x_{n-1}x_n+f_3 \qquad \text{if $p=2$ and $n$ is even}, \label{align-2}\\
&y^{p^m}+x_1^2+x_2x_3+\dots+x_{n-1}x_n+ b\cdot x_1^3+f_3 \qquad \text{if $p=2$ and $n$ is odd,} \label{align-3}
\end{align}
where $f_3\in (x_1,\dots,x_n)^3$, $b\in k^{\times},$ and $f_3$ has no $x_1^3$ term in the last case. 

An easy dimension counting argument yields the following proposition (cf.~\cite[\textsection18]{Kollar}).

\begin{proposition}\label{proposition:good-sections}
Let $W\subset H^0(X,L^{\otimes p^m})$ be such that for every closed point $x\in X$ the restriction map 
$$
W\xr{} \sO_{X,x}/m_x^{4} \otimes L^{\otimes p^m}
$$
is surjective. For a general section $s\in W$ the corresponding $p^m$ cyclic covering has non-degenerate singularities. 
\end{proposition}

\begin{remark}
If $p\neq 2$ or $\dim\, X$ even then the following surjectivity is sufficient to conclude the assertion of the proposition: 
$$
W\xr{} \sO_{X,x}/m_x^{3} \otimes L^{\otimes p^m} \quad \text{for every closed point $x\in X$.}
$$ 
\end{remark}

In order to handle the case 
$d_i=2=p, m=1,$ and $n+r+1-d'+2\leq 3$ in Theorem~\ref{thm-main-alg-closed} we need the following lemma.

\begin{lemma} \label{lemma-2-non-degenerate}
For a general complete intersection $X$ in $\P^{n+r}$ with $n\geq 2$ and multi-degree $(d_1,d_2,\dots,d_r)$ such that $d_1\geq 2$, and a general $s\in H^0(\P^{n+r},\OO(2))$, the double covering corresponding to $s_{\mid X}$ has non-degenerate singularities.
\begin{proof}
Only the case $p=2$ and $n$ odd has to be proved. Consider the variety $A$ consisting of points $(x,f_1,\dots,f_r,s)$ where $x\in \P^{n+r}$, $(f_1,\dots,f_r,s)$ are homogeneous of degree $(d_1,\dots,d_r,2)$, $X=V(f_1)\cap \dots \cap V(f_r)$ is smooth at $x$, and $d(s)_{\mid X}$ is vanishing at $x$. Those points for which the double covering corresponding to $s_{\mid X}$ has non-degenerate singularities at $x$ form an open set $B$. It is not difficult to show that it is non-empty. Indeed, take $x=[1:0:0\dots:0]$, and (in coordinates $x_1,\dots,x_{n+r}$ around $x$) $s=1+x_1^2+x_2x_3+\dots+x_{n-1}x_{n}+x_1x_{n+1}$, $f_{1}=x_{n+1}+x_1^2$, and $f_{i}=x_{n+i}+\text{terms of degree $\geq 2$}$. 

Let $V\subset A$ be the open set consisting of points such that $V(f_1)\cap \dots \cap V(f_r)$ is smooth. Since $B\cap V\neq \emptyset$, we conclude that for a general complete intersection $X$ there is an open non-empty set $U\subset X$ such that for any $x\in U$ the set 
\begin{align*}
\{s\in H^0(\P^{n+r},\OO(2))\mid &d(s)_{\mid X}(x)=0\; \text{and $s$ does not yield}\\ 
&\text{ a non-degenerate double covering at $x$}\}
\end{align*} 
has codimension $\geq n+1$. Counting dimensions yields the claim.    
\end{proof}
\end{lemma}

The following proposition has been proved for the case $m=1$ in \cite{CTP2}, and for the general case in \cite{O}. 

\begin{proposition}\label{proposition:resolution}
Suppose $Y$ has non-degenerate singularities. Then by successively blowing up singular points, we can construct a resolution of singularities $r:\tilde{Y}\xr{} Y$ such that the exceptional divisor is a normal crossings divisor  (cf.~\cite{Kollar}). Over every singular point $y\in Y$ the fiber $r^{-1}(y)$ is a chain of smooth irreducible divisors, each component of which is  either a  projective space, a  smooth quadric or a projective bundle over a smooth quadric. The intersection of two irreducible components is a smooth quadric or is empty. In particular, since $k$ is algebraically closed,
the morphism $r$ is totally $\CH_0$ trivial.
\begin{proof}
We distinguish three cases:
\begin{enumerate}
\item $p$ is odd,
\item $p=2$, and $n$ is even,
\item $p=2$, and $n$ is odd.
\end{enumerate}
 In any case we will only blow up singular points, and over any singular $s$ there will be at most one singular point appearing in the exceptional divisor of the blow up of $s$. 

We may assume that $Y$ has only one singular point. In case (1), note that we have a singularity of the form \eqref{align-1}. We need $\frac{p^m-1}{2}+1$ blow ups:
$$
\tilde{Y}:=Y_{\frac{p^m-1}{2}+1}\xr{} Y_{\frac{p^m-1}{2}} \xr{} \dots\xr{} Y_1\xr{} Y_0:=Y. 
$$  
Around the singularity of $Y_i$, for $0\leq i<\frac{p^m-1}{2}$, $Y_i$ is defined by 
\begin{equation}\label{equation:around-cone-point}
y^{p^m-2\cdot i}+{x}_1'^2+\dots+{x}_n'^2+f'_3, 
\end{equation}
where $x_i'=\frac{x_i}{y^i}$ and $f_3'\in y^i\cdot ({x}_1',\dots,{x}_n')^3$. Therefore the exceptional divisor of $Y_{i+1}\xr{} Y_i$ is the cone $C$ defined by ${x}_1'^2+\dots+{x}_n'^2$ in the projective space with homogeneous variables $y,{x}_1',\dots,{x}_n'$. For $i=\frac{p^m-1}{2}$, $Y_i$ is also given by \eqref{equation:around-cone-point} around the vertex of the exceptional divisor, hence $p^m-2i=1$ implies that it is smooth and the exceptional divisor of $Y_{\frac{p^m-1}{2}+1}\xr{} Y_{\frac{p^m-1}{2}}$ is $\P^{n-1}$. Denoting by $\tilde{E}_i$ the strict transform in $\tilde{Y}$ of the exceptional divisor of $Y_i\xr{} Y_{i-1}$, we conclude that $\tilde{E}_i$ is the blow-up of $C$ in its vertex if $i\leq \frac{p^m-1}{2}$, and $\tilde{E}_{\frac{p^m-1}{2}+1}=\P^{n-1}$. Every $\tilde{E}_i$ has only non-empty intersection with $\tilde{E}_{i+1}$ (if $i\leq \frac{p^m-1}{2}$) and $\tilde{E}_{i-1}$ (if $i>1$); the intersection is the smooth quadric given by ${x}_1'^2+\dots+{x}_n'^2$ in the projective space with homogeneous variables ${x}_1',\dots,{x}_n'$.

For case (2), this case is similar to (1). We need $2^{m-1}$ blow ups to arrive at $\tilde{Y}$.  Around the singularity of $Y_i$, for $0\leq i<2^{m-1}$, $Y_i$ is defined by 
\begin{equation}\label{equation:around-cone-point-p=2}
y^{2^m-2\cdot i}+{x}_1'x_2'+\dots+x_{n-1}'{x}_n'+f'_3, 
\end{equation}
and the exceptional divisor of $Y_{i}\xr{} Y_{i-1}$ is the cone $C$ defined by ${x}_1'x_2'+\dots+x_{n-1}'{x}_n'$ in the projective space $P$ with homogeneous variables $y,{x}_1',\dots,{x}_n'$. The exceptional divisor of $\tilde{Y}:=Y_{2^{m-1}} \xr{} Y_{2^{m-1}-1}$ is the smooth quadric defined by $y^{2}+{x}_1'x_2'+\dots+x_{n-1}'{x}_n'$ in $P$. Again, the intersection of $\tilde{E}_i$ with $\tilde{E}_{i-1}$ is the smooth quadric given by ${x}_1'x_2'+\dots+x_{n-1}'{x}_n'$ in the projective space with homogeneous variables ${x}_1',\dots,{x}_n'$.

For case (3), we need $2^{m}$ blow ups to arrive at $\tilde{Y}$. The case $m=1$ is easy to check; we will assume $m>1$. We start with $Y$ and the singularity \eqref{align-3}. After $2^{m-1}-1$ blow ups the singularity  is of the form
$$
b\cdot y^{2^{m-1}+2}+{x_1^{[1]}}^2+x'_2x'_3+\dots+x'_{n-1}x'_n+ b\cdot x_1^{[1]}\cdot y^{2^{m-1}+1} + {\rm h.o.t.},
$$
where $x'_i=\frac{x_i}{y^{2^{m-1}-1}}$, $x_1^{[1]}=x_1'+y$, and the higher order terms ${\rm h.o.t.}$ can be ignored. After $2^{m-2}$ more blow ups we introduce $x_{1}^{[2]}=\frac{x_1^{[1]}}{y^{2^{m-2}}}+\sqrt{b}\cdot y$, after $2^{m-3}$ more blow ups we introduce $x_{1}^{[3]}=\frac{x_1^{[2]}}{y^{2^{m-3}}}+\sqrt{\sqrt{b}\cdot b}\cdot y$, etc. The singularity is after $2^{m-1}-1 + 2^{m-2} + 2^{m-3} + \dots + 2^{m-i}$ blow ups of the form
\begin{equation}\label{equation-sing-3}
b_i\cdot y^{2^{m-i}+2}+{x_1^{[i]}}^2+x'_2x'_3+\dots+x'_{n-1}x'_n+ b\cdot x_1^{[i]}\cdot y^{2^{m-i}+1} + {\rm h.o.t.},
\end{equation}
where $x'_i=\frac{x_i}{y^{ -1+\sum_{j=1}^i 2^{m-j}}}$ and $b_i=b\cdot\sqrt{b_{i-1}}$ with $b_1=b$. After $2^m-2$ blow ups we get a singularity \eqref{equation-sing-3} with $i=m$. After one more blow up the variety becomes smooth, and we need one more blow up to obtain an exceptional divisor with strict normal crossings. 

The exceptional divisor $E_i$ of $Y_i\xr{} Y_{i-1}$ is a cone defined by ${x_1^{[j]}}^2+x'_2x'_3+\dots+x'_{n-1}x'_n$ in the projective space with homogeneous variables $y,x_1^{[j]},x_2',\dots,x_n'$, except for the last blow up where it is a projective space. The strict transform $\tilde{E}_i$ is the blow up of the vertex. 
\end{proof}
\end{proposition}

For $p$ odd or $n$ odd, 
we get a projective space as exceptional divisor in the last step.
 Denoting by $E$ the sum over all components of the exceptional divisor of $r$, we set
\begin{equation}\label{definition-E-prime}
  E':= \begin{cases} E  + \text{(exc.~div.~from last step)}, &\text{if $p$ is odd or $n$ is odd,} \\
    E &\text{if $p=2$ and $n$ is even.}
  \end{cases}
\end{equation}
Thus the 
 exceptional divisor of the last blow up (a projective space) 
 has multiplicity $2$ in $E'$ in the first case. If the singularity is of the form \eqref{align-1}, \eqref{align-2}, or \eqref{align-3}, then $E'$ is the restriction of ${\rm div}(y)$ to the exceptional divisor of the resolution $r$. 

\begin{lemma}\label{lemma:rational-resolution}
The resolution $r:\tilde{Y}\xr{} Y$ is rational, that is $Rr_*\OO_{\tilde{Y}}=\OO_Y$. 
\begin{proof}
We may suppose that $Y$ has only one singularity. We will show that for each $r_i:Y_i\xr{} Y_{i-1}$, we have $Rr_{i*}\OO_{Y_i}=\OO_{Y_{i-1}}$. Since $Y_{i-1}$ is normal, it suffices to prove $R^jr_{i*}\OO_{Y_i}=0$. 
We know that $r_i$ is the blown up of a point and the exceptional divisor $D$ is a cone over a smooth quadric, a smooth quadric, or a projective space, and comes with a given embedding into projective space; we call the corresponding ample line bundle $\OO_D(1)$. In any case, $H^{>0}(D,\sO(-s\cdot D))\cong H^{>0}(D,\sO(s))=0$ for all $s\geq 0$, where $\sO_D(s)=\sO_D(1)^{\otimes s}$. This implies the claim.
\end{proof}
\end{lemma}

\begin{lemma}\label{lemma:cohomology-exc-divisor}
  Let $E'$ be as defined in \eqref{definition-E-prime}. For all $i\geq 2$ we have  
  $$
  H^i(E',\sO(E'))=0.
  $$
\begin{proof}
  We may suppose that $Y$ has only one singular point.  The exceptional divisor is $\sum_{i=1}^s \tilde{E}_i$, and $\tilde{E}_i$ has non-empty intersection only with $\tilde{E}_{i+1}$ and $\tilde{E}_{i-1}$. Recall that all intersections are smooth quadrics. If $i\neq s$ then $\tilde{E}_i$ is the blow up at the vertex of a cone $C_i\subset \P^n$ over a smooth quadric $Q_i\subset \P^{n-1}$; let $r_i:\tilde{E}_i \xr{} C_i$ denote the blow up. 

  For $i=1,\dots,s-2$, we have $\OO_{\tilde{E}_i}(\tilde{E}_i + \tilde{E}_{i+1})\cong r_i^*\OO_{C_i}(-1)$, hence 
\begin{equation}\label{equation-intersections}
\OO_{\tilde{E}_i\cap \tilde{E}_{i+1}}(E')\cong \OO_{\tilde{E}_i\cap \tilde{E}_{i+1}}.
\end{equation}
For $i=2,\dots,s-2$, we obtain $\OO_{\tilde{E}_i}(E')\cong \OO_{\tilde{E}_i}$. 

  If $p$ or $n$ is odd then $\OO_{\tilde{E}_{s-1}}(\tilde{E}_{s-1} + 2\cdot \tilde{E}_{s})\cong r_{s-1}^*\OO_{C_{s-1}}(-1)$ hence $\OO_{\tilde{E}_{s-1}}(E')\cong \OO_{\tilde{E}_{s-1}}$, and $\OO_{\tilde{E}_{s}}(\tilde{E}_{s-1}+ 2\cdot \tilde{E}_{s})\cong \OO_{\P^{n-1}}$; thus \eqref{equation-intersections} holds for $i=s-1$.  If $p$ and $n$ are even then $\OO_{\tilde{E}_{s-1}}(\tilde{E}_{s-1} + \tilde{E}_{s})\cong r_{s-1}^*\OO_{C_{s-1}}(-1)$ hence $\OO_{\tilde{E}_{s-1}}(E')\cong \OO_{\tilde{E}_{s-1}}$. Moreover, we have $\OO_{\tilde{E}_{s}}(E')\cong \OO_{\tilde{E}_{s}}$. This implies the assertion easily.  
\end{proof}
\end{lemma}

\newcommand{\are}{\ar@{>>}}

\subsection{}
Again, we assume that $Y$ has non-degenerate singularities.
We denote by $U\subset X$ the complement of the critical points, $Y_{sm}=\pi^{-1}(U)$; we have
\newcommand{\drw}[2]{W_{#1}\Omega_{#2/k}^1}
$$
W_l(\pi)^*\drw{l}{U} \xr{} \drw{l}{Y_{sm}},
$$
but there is no Verschiebung on $W_l(\pi)^*\drw{l}{U}$. Therefore we define \newcommand{\imV}[1]{{\rm Im}_V(\drw{#1}{U})}
$${\rm Im}_V(\drw{l}{U})\subset \drw{l}{Y_{sm}}$$ 
inductively on $l$ by
$$\imV{l} = {\rm image}(W_l(\pi)^*\drw{l}{U} \xr{}  \drw{l}{Y_{sm}})+V(\imV{l-1}).$$
We have a $R,V,F$ calculus for $\imV{l}$, that is, morphisms
\begin{align*}
&R:\imV{l}\xr{} \imV{l-1}, \\
&V:\imV{l-1}\xr{} \imV{l},  \\ 
&F:\imV{l}\xr{} \imV{l-1},
\end{align*}
satisfying the relations induced by $\drw{*}{Y_{sm}}$ (see \cite{Illusie-drw}). By abuse of notation, any composition of maps $R$ will be also denoted by $R$.

We are going to need several statements on $\imV{l}$ in Theorem~\ref{thm:cyclic-covers} which we provide in the following.
\begin{lemma}\label{lemma-R-surjective}
The evident map 
\begin{multline}
\ker\left( R: W_l(\pi)^*\drw{l}{U} \xr{} \pi^*\Omega^1_{U}  \right) \\ \xr{} \ker\left( R: \imV{l}\xr{} \imV{1} \right)/V(\imV{l-1}).
\end{multline}
is surjective if $l\leq m$.
\begin{proof}
The target is the image of $R^{-1}(\ker\left( \pi^*\Omega_U^1 \xr{} \Omega^1_{Y_{sm}} \right))\subset W_l(\pi)^*\drw{l}{U}$ via the evident map $W_l(\pi)^*\drw{l}{U} \xr{} \imV{l}/V(\imV{l-1})$. Locally, $Y_{sm}$ is defined by $y^{p^m}-f$, for $f\in \sO_{U}$, and 
$
\ker\left( \pi^*\Omega_U^1 \xr{} \Omega^1_{Y_{sm}} \right) 
$
is generated by $d(f)$. Since $d([f])\in W_l(\pi)^*\drw{l}{U}$ is a lifting of $d(f)$ whose image vanishes in $\imV{l}$ (here we use $l\leq m$), the claim follows.
\end{proof}
\end{lemma}

Recall the subsheaves $B_n\Omega^1_{U/k}$ of $\Omega^1_{U/k}$, $n=1, 2,\ldots$ (see for example \cite[\S I.2.2]{Illusie-drw}). 
We have a short exact sequence 
$$
\drw{l-1}{U} \xr{V} \ker\left( R: \drw{l}{U} \xr{} \Omega^1_{U}  \right) \xr{F^{l-1}} B_{l-1}\Omega_U^1 \xr{} 0.
$$ 
With the appropriate $W_l(\sO_U)$-module structures this becomes a short exact sequence of $W_l(\sO_U)$-modules. We obtain the following diagram
\begin{equation}\label{equation-strategy-understand-kernel-V}
\xymatrix{
  W_l(\pi)^*\ker\left( R: \drw{l}{U} \xr{} \Omega^1_{U}  \right)/W_l(\pi)^*(V)(W_l(\pi)^*\drw{l-1}{U}) 
 \ar[r]^-{\cong}  \are[d]^{\text{surjective by Lemma~\ref{lemma-R-surjective}}}
&
 \pi^*B_{l-1}\Omega_{U}^1
\\
 \ker\left( R: \imV{l}\xr{} \imV{1} \right)/V(\imV{l-1}) 
 \ar[d]^{(*)} 
&
\\
  \ker\left( R: \drw{l}{Y_{sm}} \xr{} \Omega^1_{Y_{sm}}  \right)/V(\drw{l-1}{Y_{sm}}) 
 \ar[r]^-{\cong} 
&
 B_{l-1}\Omega_{Y_{sm}}^1.
}
\end{equation}
The induced map 
\begin{equation}\label{equation-B-l-1}
\pi^*B_{l-1}\Omega_{U}^1\xr{} B_{l-1}\Omega_{Y_{sm}}^1
\end{equation}
is the natural one, that is, given by $a\otimes \pi^{-1}(\omega)\mapsto {\rm Frob}^{l-1}(a)\cdot \pi^{-1}(\omega)$. We would like to show that $(*)$ is injective, which we prove by computing the kernel of \eqref{equation-B-l-1} and showing that it is killed in $\imV{l}$.  

It is convenient to use the isomorphism \cite[(I.3.11.4)]{Illusie-drw}
\begin{equation}\label{equation-understanding-B}
F^{l-2}d:W_{l-1}(\sO_U)/F(W_{l-1}(\sO_U)) \xr{\cong} B_{l-1}\Omega_U^1.  
\end{equation}
The $W_{l}(\sO_U)$-module structure on the left is via the Frobenius $F:W_{l}(\sO_U)\xr{} W_{l-1}(\sO_U)$.
We give $W_{l-1}(\sO_{Y_{sm}})/F(W_{l-1}(\sO_{Y_{sm}}))$ the analogous $W_{l}(\sO_{Y_{sm}})$-module structure.

\begin{lemma}\label{lemma-compute-the-kernel}
Suppose $Y_{sm}$ is defined by $y^{p^m}-f$ for $f\in \sO_{U}$ (this is the local picture).
The kernel of 
$$
W_{l}(\pi)^* \left( W_{l-1}(\sO_U)/F(W_{l-1}(\sO_U)) \right) 
\xr{} W_{l-1}(\sO_{Y_{sm}})/F(W_{l-1}(\sO_{Y_{sm}}))
$$ 
is generated by $V(W_{l-1}(\sO_{Y_{sm}}))\otimes W_{l}(\pi)^{-1}(W_{l-1}(\sO_U))$ and elements of the form
\begin{equation}\label{equation-claim-elements-in-kernel}
[y^i]\otimes \pi^{-1}(V^j(b))- [y^{i \% p^{m-1-j}}] \otimes \pi^{-1}(V^j([f^{(i:p^{m-1-j})}]\cdot b)), 
\end{equation}
for all $0\leq j\leq l-2$, $i\geq p^{m-1-j}$, and $b\in W_{l-1-j}(\sO_U)$. Here, $i \% p^{m-1-j}$ means the remainder of $i$ in the division by $p^{m-1-j}$, and $i= (i:p^{m-1-j})\cdot p^{m-1-j}+i \% p^{m-1-j}$.
\begin{proof}
The kernel contains $V(W_{l-1}(\sO_{Y_{sm}}))\otimes W_{l}(\pi)^{-1}(W_{l-1}(\sO_U))$, because $V(a)\otimes \pi^{-1}(b)$ maps to $F(V(a))\cdot \pi^{-1}(b)=pa\cdot \pi^{-1}(b)=F(V(a\cdot \pi^{-1}(b)))$. Moreover,
\begin{multline*}
[y^i]\otimes \pi^{-1}(V^j(b))- [y^{i \% p^{m-1-j}}] \otimes \pi^{-1}(V^j([f^{(i:p^{m-1-j})}]\cdot b)) \\
\mapsto  [y^{pi}]\cdot  V^j(b)- [y^{(i \% p^{m-1-j})\cdot p}] \cdot V^j([f^{(i:p^{m-1-j})}]\cdot b)\\
= V^j( ([y^{p^{1+j}\cdot i}]-[y^{(i \% p^{m-1-j})\cdot p^{1+j}}\cdot f^{(i:p^{m-1-j})}])\cdot b)=0.
\end{multline*}

In order to show that these are all elements in the kernel, we proceed by induction on $l$. First, we assume $l=2$. Without loss of generality, we 
need only 
consider elements in the kernel that are of the form $\sum_{i} [y^i]\otimes \pi^{-1}(b_i)$. By \'etale base change, we may assume that $U=\Spec(k[x_1,\dots,x_n])$ and $x_1=f$,
 hence $Y_{sm}=\Spec(k[y,x_2,\dots,x_n])$. By using elements of the form \eqref{equation-claim-elements-in-kernel}, we may suppose that $b_i=b_i(x_2,\dots,x_n)$. Since 
$\sum_i y^{ip}b_i\in k[y^p,x_2^p,\dots,x_n^p]$ implies $b_i\in k[x_2^p,\dots,x_n^p]$, we are done. 

Suppose now that $l>2$. By induction, we 
need
only consider elements in the kernel that are of the form
$$
\sum_{i} [y^i]\otimes \pi^{-1}(V^{l-2}(b_i)),
$$ 
and we may use the same argument as for the $l=2$ case.
\end{proof}
\end{lemma}

\begin{proposition} \label{proposition:injectivity}
Suppose $l\leq m$. The map 
\begin{multline*}
\ker\left( R: \imV{l}\xr{} \imV{1} \right)/V(\imV{l-1}) \xr{} \\ \ker\left( R: \drw{l}{Y_{sm}} \xr{} \Omega^1_{Y_{sm}}  \right)/V(\drw{l-1}{Y_{sm}}) 
\end{multline*}
is injective. 
\begin{proof}
In view of Diagram \eqref{equation-strategy-understand-kernel-V} and Lemma~\ref{lemma-compute-the-kernel}, we need to prove that the following elements vanish in 
$
\imV{l}/V(\imV{l-1}),
$
\begin{enumerate}
\item $V(a)\cdot dV(b)$ for $a\in W_l(\sO_{Y_{sm}})$ and  $b\in W_{l-1}(\sO_U)$,
\item $[y^i]\cdot dV^{j+1}(b) - [y^{i \% p^{m-1-j}}] \cdot dV^{j+1}([f^{(i:p^{m-1-j})}]\cdot b)$ for $b\in W_{l-1-j}(\sO_U)$.
\end{enumerate}
For (1), we have 
$$
V(a)\cdot dV(b)=V(a\cdot d(b))\in V(\imV{l-1}).
$$
For (2), we compute
\begin{align*}
[y^i]\cdot dV^{j+1}(b) &= d([y^i]\cdot V^{j+1}(b))- V^{j+1}(b)\cdot d([y^i])\\
 &=dV^{j+1}([y^{i\cdot p^{1+j}}]\cdot b) -  V^{j+1}(b)\cdot d([y^i]) \\
 &= dV^{j+1}([y^{(i \% p^{m-1-j})\cdot p^{1+j}}]\cdot [f^{(i:p^{m-1-j})}]\cdot b) -  V^{j+1}(b)\cdot d([y^i])\\
&=d \left([y^{i \% p^{m-1-j}}]\cdot V^{j+1}([f^{(i:p^{m-1-j})}]\cdot b) \right) - V^{j+1}(b)\cdot d([y^i])\\
&=V^{j+1}([f^{(i:p^{m-1-j})}]\cdot b) \cdot d([y^{i \% p^{m-1-j}}]) \\
&+ [y^{i \% p^{m-1-j}}]\cdot dV^{j+1}([f^{(i:p^{m-1-j})}]\cdot b) - V^{j+1}(b)\cdot d([y^i]),
\end{align*}
which together with
\begin{multline*}
V^{j+1}([f^{(i:p^{m-1-j})}]\cdot b) \cdot d([y^{i \% p^{m-1-j}}]) - V^{j+1}(b)\cdot d([y^i]) \\=
V^{j+1}\left(b\cdot \left([f^{(i:p^{m-1-j})}]\cdot F^{j+1}(d([y^{i \% p^{m-1-j}}])) - F^{j+1}(d([y^i]))\right) \right)\\=
V^{j+1}\left(b\cdot F^{j+1}(d([y^{(i:p^{m-1-j})\cdot p^{m-1-j}}] [y^{i \% p^{m-1-j}}] - [y^i]))\right)=0
\end{multline*}
(note that $F^{j+1}(d([y^{(i:p^{m-1-j})\cdot p^{m-1-j}}]))=0$) implies the claim. 
\end{proof}
\end{proposition}

\subsection{} 
We denote by $\jmath: r^{-1}(Y_{sm}) \xr{} \tilde{Y}$ the open immersion. We will work with the logarithmic de Rham-Witt complex 
$$
\drw{l}{\tilde{Y}}(\log E) \subset \jmath_* \drw{l}{Y_{sm}}.
$$
Locally, when $E=\cup_{i=1}^r V(f_i)$ with $V(f_i)$ smooth, $\drw{l}{\tilde{Y}}(\log E)$ is 
generated
as a $W_l(\OO_{\tilde{Y}})$ submodule of $\jmath_* \drw{l}{Y_{sm}}$ 
by $\drw{l}{\tilde{Y}}$ and $\langle \frac{d[f_i]}{[f_i]} \mid i=1,\dots, r\rangle$. As for the de Rham complex there is an exact sequence 
\begin{equation}\label{equation-residue-exact-seq}
0\xr{} \drw{l}{\tilde{Y}} \xr{} \drw{l}{\tilde{Y}}(\log E) \xr{} \bigoplus_{i=0}^r W_l(\OO_{V(f_i)})\xr{} 0.
\end{equation}
We have the usual $F,V,R$ calculus for $\drw{*}{\tilde{Y}}(\log E)$.

We define 
$$
K_l:= \jmath_* \imV{l} \cap \drw{l}{\tilde{Y}}(\log E) \subset \jmath_* \drw{l}{Y_{sm}}.
$$
We have a $F,V,R$ calculus for $K_*$ induced by the one for $\imV{*}$ and $\drw{*}{\tilde{Y}}(\log E)$. We set $Q_*:=\drw{*}{\tilde{Y}}(\log E)/K_*$. 

\begin{lemma}\label{lemma-lifting} 
Suppose that $p\neq 2$ or $n$ is even. Then, for all $l\geq 1$, the following map is surjective:
$$
R:K_l\xr{} K_1.
$$
\begin{proof}
The first case is $p\neq 2$. We need to compute $K_1$. We may assume that $Y$ has only one singularity as in the proof of Proposition \ref{proposition:resolution}. Recall that $\tilde{Y}$ is constructed as a sequence of blow ups $\dots \xr{} Y_{i} \xr{} Y_{i-1} \xr{} \dots \xr{} Y$. We denote by $r_i:Y_{i}\xr{} Y$ the evident composition; we let $D_i$ be the exceptional divisor of $r_i$, and $E_i$ denotes the exceptional divisor of $Y_i\xr{} Y_{i-1}$. We would like to understand
\begin{equation}\label{equation-compute-K1}
\jmath_{Y_i\backslash D_i,*}\left( {\rm image}\left( r_i^*\pi^*\Omega_{X}^1\mid_{Y_i\backslash D_i} \xr{}\Omega^1_{Y_i\backslash D_i} \right) \right)\cap \Omega^1_{Y_{i,{\rm sm}}} (\log D_i\mid_{Y_{i,{\rm sm}}}),
\end{equation}
in a neighborhood of $E_i\cap Y_{i,{\rm sm}}$, where $Y_{i,{\rm sm}}$ is the smooth locus of $Y_{i}$, and $\jmath_{Y_i\backslash D_i}: Y_i\backslash D_i \xr{}  Y_{i,{\rm sm}}$ is the open immersion. 

As in the proof of Proposition \ref{proposition:resolution}, we have coordinates $y,x_1',\dots,x_n'$ around the singular point of $Y_{i-1}$, where $x_j'=\frac{x_j}{y^{i-1}}$. We can cover $E_i$ by $n+1$ open sets $V_0,V_1,\dots,V_{n}$, where $V_0$ is a hypersurface in the affine space with coordinates $y,\frac{x_1'}{y},\dots,\frac{x_n'}{y}$, and $V_j$ is a hypersurface in the affine space with coordinates $\frac{y}{x_j'},\frac{x_1'}{x_j'},\dots, x_j',\dots, \frac{x_n'}{x_j'}$, for $j=1,\dots,n$. On $V_0$ we have $E_i\cap V_0=D_i\cap V_0=V(y)$. Note that if $i=\frac{p^m-1}{2}+1$, which is the last blow up, then $E_i\cap V_0$ is empty.

On $V_j$ we have $E_i\cap V_j=V(x_j')$ and $D_i\cap V_j=V(y)$ if $j=1,\dots,n$ and $i\not\in\{1,\frac{p^m-1}{2}+1\}$, that is, except for the first and the last blow up. For the first blow up $(i=1)$, we have $E_i\cap V_j=D_i\cap V_j=V(x_j')$. For the last blow up $(i=\frac{p^m-1}{2}+1)$, we have $E_i\cap V_j=V(x_j')$ and $D_i\cap V_j=V(\frac{y}{x_j'})$. 

We claim that the restriction of (\ref{equation-compute-K1}) to $V_0$ is generated by $\frac{dx_1}{y^i},\dots , \frac{dx_n}{y^i}$, and the restriction of (\ref{equation-compute-K1}) to $V_j$ is generated by $\frac{dx_1}{x_j},\dots, \frac{dx_j}{x_j}, \dots , \frac{dx_n}{x_j}$. It is obvious that all differential forms are contained in the left hand side of (\ref{equation-compute-K1}), and we need to show that they are contained in $\Omega^1_{Y_{i,{\rm sm}}} (\log D_i\mid_{Y_{i,{\rm sm}}})$. Indeed, 
$
\frac{dx_j}{y^i}= \frac{d\left(\frac{x'_j}{y}\cdot y^i\right)}{y^i}=d\left(\frac{x_j'}{y}\right)+ i \cdot \frac{x_j'}{y}\cdot \frac{dy}{y}$, 
and 
\begin{multline*}
\frac{dx_k}{x_j}=d\left(\frac{x_k}{x_j}\right)+\frac{x_k}{x_j}\cdot \frac{dx_j}{x_j}=d\left(\frac{x_k'}{x_j'}\right)+\frac{x_k'}{x_j'}\cdot \frac{dx_j}{x_j}= \\ d\left(\frac{x_k'}{x_j'}\right)+\frac{x_k'}{x_j'}\cdot \left(\frac{dx_j'}{x_j'} + (i-1)\cdot \frac{dy}{y}\right).
\end{multline*}
In order to show that the given differential forms are generators, we note that the quotient of $\Omega^1_{Y_{i,{\rm sm}}} (\log D_i\mid_{Y_{i,{\rm sm}}})\cap V_j$ by the module generated by these forms is a quotient of a free rank $=1$ module. Since the quotient of $\Omega^1_{Y_{{\rm sm}}}$ by the image of $\pi^*(\Omega^1_{U})$ is free of rank $1$, the claim follows. 

The case $p=2$ and $n$ even can be proved in the same way. 

In order to prove that $K_l\xr{} K_1$ is surjective, we may argue by induction on $i$ and only consider a neighborhood of $E_i\cap Y_{i, {\rm sm}}$ in $Y_{i, {\rm sm}}$. We note that $\frac{dx_j}{y^i}$ can be lifted by $\frac{d[x_j]}{[y^i]}\in K_l(V_0)$, and $\frac{dx_k}{x_j}$ can be lifted by $\frac{d[x_k]}{[x_j]}\in K_l(V_j)$.    
\end{proof}
\end{lemma}

\begin{remark}
We do not know whether Lemma \ref{lemma-lifting} holds if $p=2$ and $n$ is odd. We can still describe $K_1$, but the coordinate changes $x_1^{[1]},x_1^{[2]},\dots$ used in the resolution process are incompatible with the multiplicative Teichm\"uller map and evident liftings do not exist.
\end{remark}

\subsection{}
Let us assume that $p\neq 2$ or $n$ is even.  In view of the lemma, the map 
\begin{equation}\label{equation-surjectivity-to-Q-ker}
\ker(\drw{*}{\tilde{Y}}(\log E) \xr{R} \drw{1}{\tilde{Y}}(\log E)) \xr{} \ker(Q_*\xr{R} Q_1) 
\end{equation}
is surjective.

As consequence of Proposition~\ref{proposition:injectivity} we obtain the following corollary.
\begin{corollary}\label{corollary-control-kernel-V}
For all $l\leq m$, the 
composition
\begin{multline*}
\OO_{\tilde{Y}}/\OO_{\tilde{Y}}^{p^{l-1}} \xr{dV^{l-2},\cong} \ker(V:\drw{l-1}{\tilde{Y}} \xr{} \drw{l}{\tilde{Y}})\\ 
 \xr{} \ker(V:\drw{l-1}{\tilde{Y}}(\log E) \xr{} \drw{l}{\tilde{Y}}(\log E)) 
 \xr{} \ker(V:Q_{l-1}\xr{} Q_l)
\end{multline*}
is surjective on the open set $Y_{sm}$.
\begin{proof}
The first isomorphism follows from \cite[Proposition~I.3.11]{Illusie-drw}. The second arrow is an isomorphism on $Y_{sm}$. Set 
\begin{align*}
A_l&:=\ker\left( R: \imV{l}\xr{} \imV{1} \right),\\
B_l&:=\ker\left( R: \drw{l}{Y_{sm}} \xr{} \Omega^1_{Y_{sm}}  \right),\\
C_l&:= \ker\left( R: Q_{l|Y_{sm}} \xr{} Q_{1|Y_{sm}}  \right).
\end{align*}
In view of \eqref{equation-surjectivity-to-Q-ker} we have a morphism of exact sequences 
$$
\xymatrix{
0  \ar[r]
&
\imV{l-1} \ar[r] \ar[d]^{V}
&
\drw{l-1}{Y_{sm}} \ar[r] \ar[d]^{V}
&
Q_{l-1|Y_{sm}} \ar[r] \ar[d]^{V}
&
0
\\
0 \ar[r]
&
A_{l} \ar[r]
&
B_{l} \ar[r]
&
C_{l} \ar[r]
&
0,
}
$$
and the snake lemma and Proposition~\ref{proposition:injectivity}  imply the assertion.
\end{proof}
\end{corollary}

\begin{theorem}\label{thm:cyclic-covers}
Let $X$ be a smooth projective variety of dimension $n$ over an algebraically closed field of characteristic $p$. Suppose that $p$ is odd or $n$ is even. Let $L$ be a line bundle on $X$, and let $s\in H^0(X,L^{\otimes p^m})$ for $m\geq 1$. Suppose that the $p^m$ cyclic covering $\pi:Y\xr{} X$ corresponding to $s$ has only non-degenerate singularities; let $r:\tilde{Y}\xr{} Y$ be the resolution from Proposition~\ref{proposition:resolution}. Suppose that 
\begin{enumerate}
\item $n \geq 3$,
\item $H^0(X,L^{\otimes p^m}\otimes K_X)\neq 0$,
\item the Frobenius acts bijectively on $H^{n-1}(V(s),\sO)$,
\item $H^n(X,L^{\otimes -j})=0$ for all $j=0,\dots,p^{m}-1$,
\item $H^{n-1}(X,L^{\otimes -j})=0$ for all $j=0,\dots,p^{m}$.
\end{enumerate}
Then $W_m(k)\subset H^0(\tilde{Y},W_m\Omega^{n-1})$.
\end{theorem}
 
\begin{proof}
We have 
$$
{\rm coker}(\pi^*(\Omega^1_{U})\xr{} \Omega^1_{Y_{sm}})=\pi^*(L^{-1}),
$$
and this identity extends to 
$$
Q_1=r^*\pi^*(L^{-1})(E')
$$
on $\tilde{Y}$, with $E'$ as defined in \eqref{definition-E-prime}.
 If the singularity of $Y$ is of the form \eqref{align-1}, \eqref{align-2}, or \eqref{align-3}, then $Q_1$ is generated by $\frac{dy}{y}$.

In view of Lemmas~\ref{lemma:rational-resolution} and  \ref{lemma:cohomology-exc-divisor}, and conditions (2), (4), and (5), we obtain
\begin{equation}\label{equation:Q1}
H^{n-1}(\tilde{Y},Q_1)=0, \qquad H^{n}(\tilde{Y},Q_1)\cong H^n(X,L^{\otimes -p^m})\neq 0.
\end{equation}

 We will work with the short exact sequences 
\begin{align}\label{align:short-exect-seq-q-1}
0 \xr{} \ker(R:Q_l\xr{} Q_{1}) \xr{} Q_l \xr{} Q_{1} \xr{} 0,\\
\label{align:short-exect-seq-q-2}  Q_{l-1} \xr{V} \ker(R:Q_l\xr{} Q_{1}) \xr{} T_l \xr{} 0, 
\end{align}
where $T_l$ is simply defined to be the cokernel. We claim 
\begin{equation}\label{equation-claim-thm}
H^{n-1}(\tilde{Y},T_l)=0=H^n(\tilde{Y},T_l)
\end{equation}
for all $l\leq m$. The surjectivity of \eqref{equation-surjectivity-to-Q-ker} yields the surjectivity of the following composition:
\begin{multline}
\ker\left(\drw{l}{\tilde{Y}}(\log E) \xr{R} \Omega^1_{\tilde{Y}}(\log E)\right)/V\drw{l-1}{\tilde{Y}}(\log E) \\ \underset{\cong}{\xr{F^{l-1}}} B_{l-1}\Omega^1_{\tilde{Y}} \xr{} T_l
\end{multline}
\cite[page~575]{Illusie-drw}. Note that 
\begin{multline*}
\ker\left(\drw{l}{\tilde{Y}} \xr{R} \Omega^1_{\tilde{Y}}\right)/V\drw{l-1}{\tilde{Y}} \\ \xr{\cong} \ker\left(\drw{l}{\tilde{Y}}(\log E) \xr{R} \Omega^1_{\tilde{Y}}(\log E)\right)/V\drw{l-1}{\tilde{Y}}(\log E) 
\end{multline*}
is an isomorphism.

Now we need to find a complex of $W_l(\OO_{\tilde{Y}})$-modules 
$$
R_1\xr{} R_0\xr{} \ker(B_{l-1}\Omega^1_{\tilde{Y}} \xr{} T_l),
$$
such that the following conditions hold:
\begin{itemize}
\item $R_{0\mid Y_{sm}}\xr{} \ker(B_{l-1}\Omega^1_{\tilde{Y}} \xr{} T_l)_{\mid Y_{sm}}$
is surjective,
\item $H^n(\tilde{Y},R_1)\xr{} H^n(\tilde{Y},R_0)$ is surjective.  
\end{itemize}
It will follow that $H^n(\tilde{Y},T_l)=0=H^{n-1}(\tilde{Y},T_l)$. Indeed, we have $$H^n(\tilde{Y},B_{l-1}\Omega^1_{\tilde{Y}})=0=H^{n-1}(\tilde{Y},B_{l-1}\Omega^1_{\tilde{Y}})$$ by induction on $l$, and using the exact sequence \eqref{equation-ex-seq-B}. The case $l=2$ follows from assumption (4) and (5), Lemma \ref{lemma:rational-resolution}, and the short exact sequence \eqref{equation-understand-B1}. 

We take 
$$
R_{0,l}:=r^*\pi^*B_{l-1}\Omega^1_X, \quad R_{1,l}=\ker(R_{0,l}\xr{} B_{l-1}\Omega^1_{\tilde{Y}}). 
$$
Clearly, the image of $r^*\pi^*B_{l-1}\Omega^1_X$ is contained in $\ker(B_{l-1}\Omega^1_{\tilde{Y}} \xr{} T_l)$. The surjectivity of $R_{0\mid Y_{sm}}\xr{} \ker(B_{l-1}\Omega^1_{\tilde{Y}} \xr{} T_l)_{\mid Y_{sm}}$ follows from Lemma~\ref{lemma-R-surjective} and Diagram \eqref{equation-strategy-understand-kernel-V}. 

We claim that $H^n(\tilde{Y},R_{1,l})\xr{} H^n(\tilde{Y},R_{0,l})$ is surjective. We will proceed by induction on $l$. We have an exact sequence of locally free $\OO_X$-modules 
\begin{equation}\label{equation-ex-seq-B}
0\xr{} {\rm Frob}^{l-2}_*B_1\Omega^1_X \xr{} B_{l-1}\Omega^1_X \xr{C} B_{l-2}\Omega^1_X\xr{} 0,
\end{equation}
where $C$ is the Cartier operator. Therefore 
$$
0\xr{} r^*\pi^*{\rm Frob}^{l-2}_*B_1\Omega^1_X \xr{} R_{0,l} \xr{C} R_{0,l-1} \xr{} 0
$$
is exact. Lemma~\ref{lemma-compute-the-kernel} shows that $R_{1,l \mid Y_{sm}}\xr{C} R_{1,l-1 \mid Y_{sm}}$ is surjective; note that under the the isomorphism $F^{l-2}d$ from \eqref{equation-understanding-B} the Cartier operator corresponds to the restriction. By induction we need to prove that the image of
$$
  H^n(\tilde{Y},r^*\pi^*{\rm Frob}^{l-2}_*B_1\Omega^1_X) \xr{} H^n(\tilde{Y},R_{0,l})
$$
is contained in the image of $H^n(\tilde{Y},R_{1,l})$. Rationality of the resolution $r$, implies 
\begin{align*}
H^n(\tilde{Y},r^*\pi^*{\rm Frob}^{l-2}_*B_1\Omega^1_X)&=H^n(Y,\pi^*{\rm Frob}^{l-2}_*B_1\Omega^1_X)\\
&=H^n(X,{\rm Frob}^{l-2}_*(B_1\Omega^1_X) \otimes_{\OO_X} \pi_*\OO_{Y}).
\end{align*}
In view of the exact sequence 
\begin{equation}\label{equation-understand-B1}
0\xr{} \OO_X \xr{{\rm Frob}} {\rm Frob}_*\OO_X \xr{} B_1\Omega^1_X \xr{} 0, 
\end{equation}
we obtain a surjective map
$$
H^n(X,\bigoplus_{i=p^{m-l+1}}^{p^m-1} L^{-i\cdot p^{l-1}}) \xr{} H^n(\tilde{Y},r^*\pi^*{\rm Frob}^{l-2}_*B_1\Omega^1_X),
$$
because 
$$
{\rm Frob}^{l-1}_*\OO_X \otimes_{\OO_X} \pi_*\OO_{Y} = \bigoplus_{i=0}^{p^m-1} {\rm Frob}^{l-1}_*({\rm Frob}^{l-1,*}L^{-i}).
$$
For every $p^{m}> i\geq p^{m-l+1}$, we have two morphisms 
${\rm Frob}^{l-1}_*({\rm Frob}^{l-1,*}(L^{-i}))\xr{} {\rm Frob}^{l-1}_*({\rm Frob}^{l-1,*}(\pi_*\OO_Y))$; the first one is induced by ${\rm Frob}^{l-1}_*{\rm Frob}^{l-1,*}$ applied to
$
L^{-i} \subset \pi_*\OO_Y.
$
The second one is induced by ${\rm Frob}^{l-1}_*$ applied to 
\begin{align*}
{\rm Frob}^{l-1,*}(L^{-i}) = L^{-i\cdot p^{l-1}} &\xr{s^{(i:p^{m+1-l})}} L^{-(i\%p^{m+1-l})\cdot p^{l-1}}\\ &=  {\rm Frob}^{l-1,*}(L^{-(i\% p^{m+1-l})}) \\ &\xr{} {\rm Frob}^{l-1,*}(\pi_*\OO_{Y}),
\end{align*}
where the last arrow comes from $L^{-(i\% p^{m+1-l})} \subset \pi_*\OO_{Y}$. Note that after application of $H^n(X,-)$ this map vanishes, because it factors over $$H^n(X,L^{-(i\%p^{m+1-l})\cdot p^{l-1}})=0.$$ 
Subtracting the two maps yields a morphism
$$
r^*\pi^* {\rm Frob}^{l-1}_*({\rm Frob}^{l-1,*}(L^{-i}))\xr{} (R_{1,l}\cap r^*\pi^*{\rm Frob}^{l-2}_*B_{1}\Omega^1_X) 
$$
which shows that the $H^n(X,L^{-i\cdot p^{l-1}})$ piece of $H^n(\tilde{Y},r^*\pi^*{\rm Frob}^{l-2}_*B_1\Omega^1_X)$ is contained in the image of $H^n(\tilde{Y},R_{1,l})$. This proves claim \eqref{equation-claim-thm}.

In view of the short exact sequences \eqref{align:short-exect-seq-q-1}, \eqref{align:short-exect-seq-q-2}, Corollary~\ref{corollary-control-kernel-V}, vanishing of $H^n(\tilde{Y},\OO_{\tilde{Y}}/\OO_{\tilde{Y}}^{p^{l-1}})$, and \eqref{equation-claim-thm}, we obtain, for all $l\leq m$, a short exact sequence
\begin{equation}\label{equation-V-R-exact-sequence}
0\xr{} H^n(\tilde{Y},Q_{l-1}) \xr{V} H^n(\tilde{Y},Q_{l}) \xr{R} H^n(\tilde{Y},Q_1) \xr{} 0. 
\end{equation}
This enables us to define
$$
\psi_{l-1}: H^n(\tilde{Y},Q_1) \xr{} H^n(\tilde{Y},Q_1), \quad a\mapsto F^{l-1}(R^{-1}(a)).
$$
It is evident that $\psi_{l-1}=\psi_1^{l-1}$. In view of \eqref{equation:Q1} we have
$$
H^n(\tilde{Y},Q_1)\cong H^n(X,L^{-p^m}).
$$
Via this identification, the map $\psi_{1}$ is given by 
$$
H^n(X,L^{-p^m}) \xr{}  H^n(X,L^{-p^{m+1}}) \xr{\cdot s^{p-1}} H^n(X,L^{-p^{m}}), 
$$
where the first arrow is induced by the $p$-th power map $L^{-p^m} \xr{} L^{-p^{m+1}}, a\mapsto a^p$. Indeed, denoting $\imath:L^{-p^m}\xr{} \pi_*r_* Q_1$ the evident map, we have a commutative diagram
$$
\xymatrix{
L^{-p^m} \ar[rr]^{()^p} \ar[d]^{\imath}
&
&
L^{-p^{m+1}}\ar[r]^{s^{p-1}}
&
L^{-p^m} \ar[d]^{\imath}
\\
\pi_*r_* Q_1\ar[rr]^-{\pi_*r_*(F\circ R^{-1})}
&
&
\pi_*r_*\left( \frac{Q_1}{{\rm image}\left( B_1\Omega^1_{\tilde{Y}}\right)} \right)
&
\frac{\pi_*r_*Q_1}{{\rm image}\left( \pi_*r_* B_1\Omega^1_{\tilde{Y}} \right)}. \ar[l] 
}
$$
Moreover, $\psi_1$ equals the composition 
\begin{align*}
 H^n(\tilde{Y},Q_1) &\xr{=} H^n(X,\pi_*r_* Q_1) \xr{\pi_*r_*(F\circ R^{-1})} H^n\left(X,\pi_*r_*\left( Q_1/{\rm image}\left( B_1\Omega^1_{\tilde{Y}}\right) \right)\right) \\
                   &\xr{}  H^n\left(\tilde{Y},Q_1/{\rm image}\left( B_1\Omega^1_{\tilde{Y}}\right) \right) \xr{\cong}  H^n(\tilde{Y},Q_1),
\end{align*}
where the last morphism is the inverse of the projection $H^n(\tilde{Y},Q_1) \xr{\eta} H^n(\tilde{Y},Q_1/{\rm image}( B_1\Omega^1_{\tilde{Y}}))$, which is injective, because $H^n(\tilde{Y},Q_1) \xr{V} H^n(\tilde{Y},Q_2)$ factors through $\eta$.

In the notation of \cite[Definition~1.3.1]{Ch-Frobenius}, we therefore get 
$$
H^n_c(X\backslash V(s),\OO)_s \cong \bigcap_{i\geq 1} {\rm image}(\psi_1^i).
$$ 
Since $H^{n-1}(X,\OO_X)=0=H^n(X,\OO_X)$, \cite[\textsection1.4]{Ch-Frobenius} implies
$$
H^n_c(X\backslash V(s),\OO)_s \cong H^{n-1}(V(s),\OO)_s=\bigcap_{i\geq 1} {\rm image}({\rm Frob}^i).
$$

By using assumption (3), we obtain 
\begin{equation}\label{equation-Hn-Ql}
H^n(\tilde{Y},Q_{l})\cong \bigoplus_{i=1}^h W(k)/p^l,
\end{equation} 
where $h=\dim_k H^n(X,L^{-p^m})$. Indeed, since the Frobenius acts bijectively on $H^{n-1}(V(s),\OO)\cong H^{n}(X,L^{-p^m})$, $\psi_1$ is bijective on $H^n(\tilde{Y},Q_1)$. In view of \eqref{equation-V-R-exact-sequence}, any lifting of a basis of $H^n(\tilde{Y},Q_1)$ via the map $R:H^n(\tilde{Y},Q_l)\xr{} H^n(\tilde{Y},Q_1)$ will be a $W(k)/p^l$-basis of $H^n(\tilde{Y},Q_l)$.

Finally, let us show that $W_l(k)\subset H^0(\tilde{Y}, W_l\Omega^{n-1}_{\tilde{Y}})$. In view of \eqref{equation-Hn-Ql}, there is a surjective morphism of $W(k)$-modules
$$
H^n(\tilde{Y},\drw{l}{\tilde{Y}}(\log E)) \xr{} W(k)/p^l=W_{l}(k).
$$
From the residue short exact sequence \eqref{equation-residue-exact-seq} we obtain a surjective map 
$$
H^n(\tilde{Y},\drw{l}{\tilde{Y}}) \xr{} W_{l}(k).
$$
Ekedahl duality \cite{E} implies
$$
R\Gamma(W_l\Omega^{n-1}_{\tilde{Y}}) \xr{\cong} R\Hom_{W_l(k)}(R\Gamma (W_l\Omega^{1}_{\tilde{Y}}),W_l(k)[-n]),
$$ 
hence the claim. 
\end{proof}
 
\begin{remark}\label{remark-m=1}
Even for the case $m=1$ the approach is dual to the one in \cite{Kollar}. With the notation in the proof of Theorem \ref{thm:cyclic-covers}, we show that the composition 
$$
H^n(\tilde{Y},\Omega^1_{\tilde{Y}}) \xr{} H^n(\tilde{Y},Q_1) \xr{\cong}  H^n(X,L^{\otimes -p^m})
$$
is surjective. For the last isomorphism we use $n\geq 3$, because we need to use Lemma \ref{lemma:cohomology-exc-divisor}, where vanishing holds for $i>1$ only. Since we don't use Lemma \ref{lemma-lifting} for this part, the argument also works for $p=2$ and $n$ odd. Taking duals we obtain an inclusion 
$$
H^0(X,\omega_X\otimes L^{\otimes p^m})\subset H^0(\tilde{Y}, \Omega^{n-1}_{\tilde{Y}}).
$$
This corresponds to a result about extending $n-1$-forms from $Y_{{\rm sm}}$ to $\tilde{Y}$ in \cite{Kollar} (and \cite{CTP2}, \cite{O}).
\end{remark}

\end{document}